\documentclass[11pt]{amsart}
\usepackage[all]{xy}
\usepackage{latexsym}
\usepackage{amsmath}
\usepackage[T1]{fontenc}
\usepackage{amssymb,amscd}
\usepackage{setspace}
\usepackage{mathdots}
\usepackage[mathscr]{eucal}
\usepackage{bbm}
\usepackage{stmaryrd}
\usepackage{color}
\usepackage[colorlinks,linkcolor=blue]{hyperref}
\usepackage{hyperref}
\usepackage{mathrsfs}
\usepackage{tabularx}
\usepackage{tikz-cd}
\usepackage{fancyhdr}

\makeatletter
\renewcommand\@makefnmark{\mbox{\textsuperscript{\normalfont(\@thefnmark)}}}
\makeatother
 
\newtheorem{theorem}{Theorem}[section] 
\newtheorem{lemma}[theorem]{Lemma}     
\newtheorem{corollary}[theorem]{Corollary}
\newtheorem{proposition}[theorem]{Proposition}
\newtheorem{remark}[theorem]{Remark}
\newtheorem{definition}[theorem]{Definition}

\numberwithin{equation}{section}

\newcolumntype{C}{>{$}c<{$}}

\newcommand{\fa}{\mathfrak{a}}
\newcommand{\p}{\mathfrak{p}}
\newcommand{\q}{\mathfrak{q}}
\newcommand{\fm}{\mathfrak{m}}

\newcommand{\cC}{\mathcal{C}}
\newcommand{\cF}{\mathcal{F}}
\newcommand{\cO}{\mathcal{O}}

\newcommand{\cH}{\mathcal{H}}

\newcommand{\cS}{\mathcal{S}}

\newcommand{\F}{\mathbb F}

\newcommand{\Q}{\mathbb Q}
\newcommand{\R}{\mathbb R}
\newcommand{\T}{\mathbb T}
\newcommand{\C}{\mathbb C}
\newcommand{\Z}{\mathbb Z}

\newcommand{\V}{\mathbf{V}}
\newcommand{\A}{\mathbb A}
\newcommand{\bP}{\mathbb{P}}

\newcommand{\CO}{\mathfrak{C}(\cO)}

\newcommand{\et}{\mathrm{\acute{e}t}}

\newcommand{\checkV}{\check{\mathbf{V}}}

\newcommand{\GL}{\mathrm{GL}}
\newcommand{\SL}{\mathrm{SL}}
\newcommand{\JL}{\mathrm{JL}}

\newcommand{\Sym}{\mathrm{Sym}}

\providecommand{\cInd}{\mathrm{c}\textrm{-}\mathrm{Ind}}

\newcommand{\ModGfl}[1][\cO]{\mathrm{Mod}_{G_{\Q_p}}^{\mathrm{fin}}(#1)}
\newcommand{\Modfl}[1][\cO]{\mathrm{Mod}_{G,\zeta}^{\mathrm{fin}}(#1)}

\newcommand{\ModadmF}[1][\F]{\mathrm{Mod}_{G,\zeta}^{\mathrm{adm}}(#1)}
\newcommand{\Modsm}[1][\cO]{\mathrm{Mod}_{G,\zeta}^{\mathrm{sm}}(#1)}
\newcommand{\Modlfin}[1][\cO]{\mathrm{Mod}_{G,\zeta}^{\mathrm{l.fin}}(#1)}
\newcommand{\Modladm}[1][\cO]{\mathrm{Mod}_{G,\zeta}^{\mathrm{l.adm}}(#1)}

\newcommand{\Banadm}[1][E]{\mathrm{Ban}^{\mathrm{adm}}_{G,\zeta}(#1)}
\newcommand{\Modfg}[1][\cO]{\mathrm{Mod}_{G,\zeta}^{\mathrm{fg\,aug}}(#1)}

\newcommand{\onealg}{1\text{-}\mathrm{lalg}}
\newcommand{\Shat}{\check{\mathcal S}}
\newcommand{\Modproaug}[1][\cO]{\Mod_G^{\rm pro\, aug}(\cO)}

\newcommand{\BanDadm}[1][E]{\mathrm{Ban}^{\mathrm{adm}}_{D^{\times},\zeta}(#1)}

\DeclareMathOperator{\End}{{\mathrm{End}}}
\DeclareMathOperator{\Ext}{{\mathrm{Ext}}}

\DeclareMathOperator{\Gal}{{\mathrm{Gal}}}
\DeclareMathOperator{\gr}{{\mathrm{gr}}}
\DeclareMathOperator{\Hom}{{\mathrm{Hom}}}
\DeclareMathOperator{\HT}{{\mathrm{HT}}}

\DeclareMathOperator{\Ind}{{\mathrm{Ind}}}

\DeclareMathOperator{\Irr}{{\mathrm{Irr}}}

\DeclareMathOperator{\Mod}{\mathrm{Mod}}
\DeclareMathOperator{\Nrd}{\mathrm{Nrd}}

\DeclareMathOperator{\Sp}{{\mathrm{Sp}}}
\DeclareMathOperator{\Spec}{{\mathrm{Spec}}}

\def\a{\alpha}
\def\b{\beta}
\def\g{\gamma}

\def\d{\delta}
\def\D{\Delta}
\def\e{\varepsilon}

\def\l{\lambda}
\def\L{\Lambda}
\def\m{\mu}
\def\o{\omega}

\def\s{\sigma}

\def\z{\zeta}

\newcommand{\quash}[1]{}

\begin{document}

\title{On the finite length of some $p$-adic representations of the quaternion algebra over $\Q_p$}

\author{Hao LIU} 
\author{Haoran WANG}
 
\date{\today}

\maketitle

\begin{abstract}
Let $D$ be the non-split quaternion algebra over $\Q_p$. We prove  that a class of admissible unitary Banach space representations of $D^{\times}$ are topologically of finite length. 
\end{abstract}

\setcounter{tocdepth}{1}
\tableofcontents

\section{Introduction}
Let $p$ be a prime number. Let $K$ be a finite extension of $\Q_p$ and let $E$ be a sufficiently large finite extension of $K$ with ring of integers $\cO$, a fixed uniformizer $\varpi$ and residue field $\F$. Let $D$ be the central division algebra over $K$ with invariant $1/n$. For any admissible smooth
representation $\pi$ of $\GL_n(K)$ over $\cO$-torsion modules, Scholze in \cite{Scholze} constructs a Weil-equivariant sheaf $\cF_{\pi}$ on $(\bP^{n-1}_{\C_p})_{\et}$. The cohomology groups
\begin{equation*}
\cS^i(\pi)=H_{\et}^i(\bP_{\C_p}^{n-1},\cF_{\pi}),\,i\ge0,
\end{equation*}
are admissible $D^{\times}$-representations and carry a  commuting continuous $G_K$-action. Here $G_K:=\Gal(\overline{K}/K)$. Pa{\v{s}}k\={u}nas extends Scholze's functor $\{\cS^i\}_{i\ge0}$ to the category of admissible unitary Banach space representations of $\GL_n(K)$ in \cite{Paskunas-JL}. More precisely, if $\Pi$ is an admissible unitary Banach space representation of $\GL_n(K)$ and $\Theta$ is some (equivalently any) open bounded $\GL_n(K)$-invariant $\cO$-lattice in $\Pi$, then we define 
\[
\Shat^i(\Pi):=(\varprojlim_n\cS^i(\Theta/\varpi^n))_{\rm tf}\otimes_{\cO} E,
\]
where the subscript {\rm tf} means taking the maximal Hausdorff torsion-free quotient. It is expected that Scholze’s functor realizes both $p$-adic local Langlands and Jacquet--Langlands correspondences.

Let $n=2$ and $K=\Q_{p}$. The $p$-adic local Langlands correspondence has been established in this case (see \cite{Co}, \cite{PaskunasIHES} and \cite{MR3272011}). Let $\rho\colon G_{\Q_p} \to\GL_2(E)$ be an absolutely irreducible continuous representation. We write $\Pi(\rho)$ for the associated unitary admissible Banach space representation of $\GL_2(\Q_p)$. It follows from \cite[Theorem 7.8]{Scholze} that $\Shat^1(\Pi(\rho))$ is residually of infinite length. In \cite{Paskunas-JL} Pa{\v{s}}k\={u}nas shows that $\Shat^1(\Pi(\rho))$ is of finite length in the category
of admissible unitary $E$-Banach space representations of $D^{\times}$ if and only if it has finitely many finite dimensional irreducible subquotients. Dospinescu, Pa{\v{s}}k\={u}nas and Schraen  in \cite{DPS-Crelle} prove that $\Shat^1(\Pi(\rho))$ is topologically of finite length when the difference of the Hodge--Tate--Sen weights of $\rho$ is not a non-zero integer. It is proved in \cite{HW-JL2} that $\Shat^1(\Pi(\rho))$ is topologically of finite length when $\rho$ has a ``global origin'' and $\overline{\rho}$ is sufficiently generic. The following is the main result of this article.
\begin{theorem}\label{thm::intro-1}
Assume $p\ge5$. Let $\rho\colon G_{\Q_p}\to\GL_2(E)$ be a continuous absolutely irreducible representation. Suppose
\hfill\begin{enumerate}
\item $\overline{\rho}$ is absolutely irreducible and is generic in the sense of Definition \ref{generic} or
\item $\overline{\rho}^{\rm ss}\cong\chi_1\oplus\chi_2$ with $\chi_1\chi_2^{-1}|_{I_{\Q_p}}$ non-trivial.  
\end{enumerate}
Then $\check{\mathcal{S}}^1(\Pi(\rho))$ is infinite dimensional and is topologically of finite length.
\end{theorem}
This theorem generalizes \cite[Theorem 1.1]{HW-JL2} in two directions. Firstly, in \emph{loc.~cit.} it's assumed that $\rho\cong r|_{G_{F_{\p}}}$ for some promodular Galois representation (\cite[Definition 7.3.15]{MR2251474}) $r\colon G_{F}\to \GL_2(E)$, where $F$ is a totally real number field with $F_{\p}\cong\Q_p$ for some finite place $\p$ above $p$. We remove this ``global origin'' condition in Theorem \ref{thm::intro-1}. Secondly, we further consider some non-generic cases, i.e., the case $\overline{\rho}^{\rm ss}\cong\o\oplus1$ in Theorem \ref{thm::intro-1} is new.

We believe that the $p$-adic Scholze's functors should preserve finite length (at least in the case where $n=2$ and $K=\Q_{p}$) and the following theorem gives another evidence.
\begin{theorem}\label{thm::intro-2}
Let $\Pi\cong(\Ind_B^G\d_2\e\otimes\d_1)_{\rm cont}$ with unitary characters $\d_1,\d_2\colon\Q^{\times}_p\to E^{\times}$ such that  $\d_1\d_2^{-1}|_{\Z_p^{\times}}\not\equiv1\pmod\varpi$. Then $\Shat^1(\Pi)$ is  infinite dimensional and is topologically of finite length.
\end{theorem}

Let us now sketch the proof of Theorem \ref{thm::intro-1} (the proof of Theorem \ref{thm::intro-2} is similar). One of the ingredients is the Taylor--Wiles--Kisin patching method. Denote by $\e$ the $p$-adic cyclotomic character. Fix a continuous character $\psi \colon G_{\Q_p}\to \cO^{\times}$ such that $\psi\equiv\e\det\rho\pmod\varpi$. Let $R_{\overline{\rho}}^{\Box,\psi\e^{-1}}$ be the universal framed deformation ring corresponding to liftings of $\overline{\rho}$ with determinant $\psi\e^{-1}$. And let 
\[
\rho^{\Box}\colon G_{\Q_p}\to\GL_2(R_{\overline{\rho}}^{\Box,\psi\e^{-1}})
\] be a universal lifting of $\overline{\rho}$. Using the modification of the Taylor--Wiles--Kisin patching method in \cite{CEGGPS1}, \cite{CEGGPS2} constructs an $\cO[\GL_2(\Q_p)]$-module with an arithmetic action of the ring $R_{{\overline{\rho}}}^{\Box}[\![x_1,\ldots,x_g]\!]$,
 where $R_{{\overline{\rho}}}^{\Box}$ is the universal framed deformation ring of $\overline{\rho}$. By the same construction carried out in the setting of quaternionic Shimura sets and Shimura curves we obtain the patched modules $M_{\infty}$ and $L_{\infty}$, $L'_{\infty}$ respectively. We also obtain a complete noetherian local ring $R_{\infty}$ faithfully flat over $R_{\overline{\rho}}^{\Box,\psi\e^{-1}}$. The patching module $M_{\infty}$ is an $R_{\infty}[\GL_2(\Q_p)]$-module finitely generated over the completed group algebra $R_{\infty}[\![\GL_2(\Z_p)]\!]$. And $L_{\infty}$ as well as $L'_{\infty}$ are $R_{\infty}[D^{\times}]$-modules finitely generated over $R_{\infty}[\![\cO_D^{\times}]\!]$. If $x\colon  R_{\infty}\to \cO$ is a continuous $\cO$-algebra homomorphism, then
\[
\Pi_{M,x}:=\Hom_{\cO}^{\rm cont}(M_{\infty}\otimes_{R_{\infty},x}\cO,E)\]
and
\[ \Pi_{L,x}:=\Hom_{\cO}^{\rm cont}(L_{\infty}\otimes_{R_{\infty},x}\cO,E),\; \Pi_{L',x}:=\Hom_{\cO}^{\rm cont}(L'_{\infty}\otimes_{R_{\infty},x}\cO,E)\]
are admissible unitary $E$-Banach space representations of $\GL_2(\Q_p)$ and $D^{\times}$ respectively. The composition $y\colon R^{\Box,\psi\e^{-1}}\to  R_{\infty}\xrightarrow{x}E$ defines a continuous Galois representation $r_y \colon G_{\Q_p}\to \GL_2(E)$. Using \cite[Theorem 7.1]{pavskunas2021finiteness} and the technique developed in \cite[\S6.5]{CEGGPS2}, one can prove that $\Pi_{M,x}\cong\Pi(r_y)^{\oplus d}$ for some integer $d\ge1$ (see also \cite[Proposition 3.9]{tung2021automorphy} and the proof of \cite[Corollary 8.16]{DPS-Crelle}). And \cite[Theorem 8.10]{DPS-Crelle} shows that $\Shat^1(\Pi_{M,x})$ is a closed subspace of $\Pi_{L,x}$ with a finite dimensional cokernel.  We also prove
 \begin{proposition}\label{pro::intro-1}
There is an $R_{\infty}[G_{F_{\p}}\times B_{\p}^{\times}]$-equivariant isomorphism 
  \[
 L_{\infty}\cong (\rho^{\Box})^*\boxtimes_{R^{\Box,\psi\e^{-1}}_{\p}}L_{\infty}'.
 \]
 \end{proposition}
As a corollary, we have
\begin{corollary}
Let $\rho\colon G_{\Q_p}\to\GL_2(E)$ be a continuous absolutely irreducible representation. Then $\Shat^1(\Pi(\rho))$ is $\rho$-typic, i.e., there exists a unitary Banach representation ${\rm JL}(\rho)$ of $D^{\times}$ and a $G_{\Q_p}\times D^{\times}$-equivariant isomorphism
\[
\Shat^1(\Pi(\rho))\cong \rho\boxtimes {\rm JL}(\rho).
\]
\end{corollary}\label{cor::intro-1}
Let $D^{\times,1}$ be the subgroup of $D^{\times}$ of elements with reduced norm equal to $1$. Let $\check{\mathcal{S}}^1(\Pi(\rho))^{\onealg}$ be the subspace of $\check{\mathcal{S}}^1(\Pi(\rho))$ consisting of  locally algebraic vectors for the $D^{\times,1}$-action. Using the local-global compatibility and  \cite[Proposition 6.13]{Paskunas-JL}, we have the following result, which improves \cite[Proposition 6.15]{Paskunas-JL}.
\begin{theorem}\label{thm::intro-3}
Let $\rho\colon G_{\mathbb{Q}_p}\to \GL_2(E)$ be an absolutely irreducible continuous representation.  Then $\check{\mathcal{S}}^1(\Pi(\rho))^{\onealg}$ is finite dimensional. 
\end{theorem}
\begin{remark}
Pa{\v{s}}k\={u}nas in \cite[\S 1.2]{Paskunas-JL} sketches a proof of Theorem \ref{thm::intro-3} which uses the patching method to reduce to the case of \cite[Proposition 6.15]{Paskunas-JL}. However our proof of Theorem \ref{thm::intro-2} is different from that in \cite[\S 1.2]{Paskunas-JL} since we do not use \cite[Proposition 6.15]{Paskunas-JL}.
\end{remark}
 Another ingredient is the finiteness criterion established in \cite{HW-JL2}, see Theorem \ref{fincri} below. We will apply this finiteness criterion to ${\rm JL}(\rho)$ and then finish the proof of Theorem \ref{thm::intro-1}.

Let us now describe how this article is organized. In \S\ref{p-adic langlands} we recall the $p$-adic local Langlands correspondence for $\GL_2(\Q_p)$. In \S\ref{Scholze} we recall some vanishing results of Scholze's functor. In \S\ref{global} we use the Taylor--Wiles--Kisin patching method to prove Theorem \ref{thm::intro-3} and Proposition \ref{pro::intro-1}. In \S\ref{local} we review the finiteness criterion of \cite{HW-JL2}. In \S\ref{main} we prove Theorem \ref{thm::intro-1} and Theorem \ref{thm::intro-2}.

\subsection{Notation}
We fix a prime number $p\geq 5$. Let $E$ be a finite extension of $\Q_p$, with ring of integers $\cO$ and residue field $\F$. Fix a uniformizer $\varpi$ of $E$. We will assume that $E$ and $\F$ are sufficiently large.

If $F$ is a field, let $G_F:=\Gal(\overline{F}/F)$ denote its absolute Galois group. Let $\varepsilon$ denote the $p$-adic cyclotomic character of $G_F$, and $\omega$ the mod $p$ cyclotomic character. 
 
Let $V/E$ be a potentially semi-stable representation of $G_{\Q_p}$. The Hodge--Tate weight of $V$ is the multiset in which $i$ appears with multiplicity $\dim_{E} (V\otimes_{\Q_p} \C_p(i))^{G_{\Q_p}}$. For example $\HT(\e)=\{-1\}$.

\subsection{Acknowledgements}
We thank Yongquan Hu for helpful discussions during the preparation of the paper, and for his comments and suggests on an earlier draft. We thank Vytautas Pa{\v{s}}k{\={u}}nas for
answering our questions. This work is supported by National Key R\&D Program of China 2023YFA1009702 and
National Natural Science Foundation of China Grants 12371011.

\section{The $p$-adic local Langlands correspondence for $\GL_2(\Q_p)$}\label{p-adic langlands}
Let $G=\GL_2(\Q_p)$ and let $Z$ be the center of $G$. Let $K=\GL_2(\Z_p)$. We say a $G$-representation $M$ over $\cO$ is \textit{smooth} if $M=\cup_{H,i}M^{H}[\varpi^i]$, where $H$ runs through all the open subgroups of $G$. For a fixed  character $\zeta\colon Z\to\cO^{\times}$, We denote by $\Modsm$ the category of smooth $G$-representations $M$ over $\cO$ with central character $\zeta$, i.e. $Z$ acts on $M$ by $\zeta$. According to \cite[Theorem 34]{BL} and \cite[Th\'eor\`eme 1.1]{BreuilCompositio}, the absolutely irreducible objects in $\Modsm$ fall into four disjoint classes:
\hfill\begin{enumerate}
\item $1$-dimensional representations $\chi\circ\det$ with $\chi\colon\Q_p^{\times}\to\F^{\times}$ a continuous character;
\item principal series representations $\Ind_B^G(\chi_1\otimes\chi_2)$ with $\chi_i\colon\Q_p^{\times}\to\F^{\times}$ ($i=1,2$) continuous characters and $\chi_1\neq\chi_2$;
\item special series $\Sp\otimes\chi\circ\det$, where $\Sp$ is the Steinberg representation fitting into the (non-split) exact sequence $0\to\mathbf{1}\to\Ind_B^G\mathbf{1}\to\Sp\to0$;
\item supersingular representations $\pi(r,0,\chi)=(\chi\circ\det)\otimes(\cInd_{KZ}^G\Sym^r \F^2/T)$, where $0\le r\le p-1$ and $T$ is the generator of the algebra $\End_{G}(\cInd_{KZ}^G\Sym^r \F^2)$ defined in \cite[Proposition 8]{BL}.
\end{enumerate}

\subsection{Extensions and blocks}\label{block}
Let $\Irr_{G,\zeta}(\F)$ be the set of equivalence classes of smooth irreducible $\F$-representations of $G$ with central character $\zeta$. For $\pi,\pi'\in\Irr_{G,\zeta}(\F)$, We say $\pi\leftrightarrow\pi'$ if $\pi\cong\pi'$ or $\Ext^1_{G,\zeta}(\pi,\pi')\neq0$, or $\Ext^1_{G,\zeta}(\pi',\pi)\neq0$. We say $\pi$ and $\pi'$ are in the same block if there exist $\pi_1,\ldots,\pi_k\in\Irr_{G,\zeta}(\F)$ such that $\pi\leftrightarrow\pi_1$, $\pi_1\leftrightarrow\pi_2$, $\ldots$ , $\pi_k\leftrightarrow\pi'$. Clearly lying in the same block defines an equivalence relation on $\Irr_{G,\zeta}(\F)$. 

Let $\pi\in\Irr_{G,\zeta}(\F)$ be absolute irreducible. When $p\ge5$, following \cite[VII]{Co}, \cite[\S4]{EmertonOrd2}  and \cite{Paskunas-extension}, we have the following description of all the possible blocks $\mathfrak{B}$ containing $\pi$ (cf. \cite[Proposition 5.42]{PaskunasIHES}):
\hfill\begin{enumerate}
\item[(I)] If $\pi$ is supersingular, then $\mathfrak{B}=\{\pi\}$ and $\dim_{\F}\Ext^1_{G,\zeta}(\pi,\pi)=3$ (\cite[Theorem 10.13]{Paskunas-extension});
\item[(II)] If $\pi\cong \Ind_B^G\chi_1\otimes\chi_2\omega^{-1}$ with $\chi_1\chi_2^{-1}\neq1, \omega^{\pm1}$ and $\chi_1\chi_2=\z\o$ then
\begin{equation*}
\mathfrak{B}= \{\Ind_B^G\chi_1\otimes\chi_2\omega^{-1},\Ind_B^G\chi_2\otimes\chi_1\omega^{-1}\}
\end{equation*}
and we summarize the dimension of $\Ext^1_{G,\zeta}(\pi',\pi)$ in the following table (\cite[Theorem 11.5]{Paskunas-extension}):
\begin{table}[h]
\begin{tabular}{C|CC}
\pi'\backslash\pi                                                                                                                                   & \Ind_B^G\chi_1\otimes\chi_2\omega^{-1} & \Ind_B^G\chi_2\otimes\chi_1\omega^{-1}  \\
\hline
\Ind_B^G\chi_1\otimes\chi_2\omega^{-1} & 2                                                                                                                                                        & 1                                                                                                                                                                                                                                                         \\
\Ind_B^G\chi_2\otimes\chi_1\omega^{-1} & 1                                                                                                                                                        & 2                                                                                                                                                                                                                                                                                                                                                                                                                                             
\end{tabular}
\end{table}
\item[(III)] If $\pi=\Ind_B^G\chi\otimes\chi\omega^{-1}$ with $\chi^2=\z\o$ then $\mathfrak{B}=\{\pi\}$ and $\dim_{\F}\Ext^1_{G,\zeta}(\pi,\pi)=2$ (\cite[Theorem 11.5]{Paskunas-extension});
\item[(IV)] Otherwise, $\mathfrak{B}=\{\chi\circ\det,\chi\circ\det\otimes \Sp,\chi\circ\det\otimes\pi_{\a}\}$ with $\pi_{\a}=\Ind_B^G\omega\otimes\omega^{-1}$ and $\chi^2=\z$. We have the following table for $\dim_{\F}\Ext^1_{G,\z}(\pi',\pi)$ (\cite[Theorem 11.4]{Paskunas-extension}):
\begin{table}[h]
\begin{tabular}{C|CCC}
\pi'\backslash\pi & \mathbf{1} & \Sp & \pi_{\alpha} \\
\hline
\mathbf{1}                                                              & 0 & 2                  & 0                                             \\
\Sp                                             & 1 & 0                  & 1                                             \\  
 \pi_{\alpha}                  & 1 & 0                  & 2                                            
\end{tabular}
\end{table}
\end{enumerate}

\subsection{Colmez's Montr\'eal functor}
Let $\ModGfl$ be the category of continuous finite-length $G_{\Q_p}$-representations on $\cO$-torsion modules and let $\Modfl$ be the full subcategory of $\Modsm$ consisting of finite-length objects. Colmez \cite{Co} has defined a covariant exact functor $\V\colon \Modfl\to\ModGfl$.  If $\chi\colon \Q^{\times}_{p}\to \cO^{\times}$ is a unitary continuous character, then $\V(\pi\otimes \chi\circ\det)=\V(\pi)\otimes\chi$, where we also view $\chi$ as a continuous character of $G_{\Q_p}$ via the local class field theory. We list the values of $\V$ on the absolutely irreducible $G$-representations over $\F$:
\hfill\begin{enumerate}
\item $\V(\chi\circ\det)=0$,
\item $\V(\Sp\otimes\chi\circ\det)=\omega\chi$,
\item $\V(\Ind_{B}^{G}\chi_1\otimes\chi_2\omega^{-1})=\chi_2$,
\item $\V(\pi(r,0,\chi))=(\Ind_{G_{\Q_{p^2}}}^{G_{\Q_p}}\omega_2^{r+1})\otimes\chi\g_{\sqrt{-1}}$.
\end{enumerate}
Here $\o_2$ is Serre's fundamental character of level $2$ given by (\ref{w_2}) and $\g_{\sqrt{-1}}$ is the unramified character of $G_{\Q_{p}}$ sending the arithmetic Frobenius to $\sqrt{-1}$.

Let $A$ be a complete noetherian local $\cO$-algebra. By a profinite augmented representation of $G$ over $A$, we mean a profinite $A$-module with an $A$-linear $G$-action and a jointly continuous $A[\![H]\!]$-action for some (equivalently any) compact open subgroup $H$ such that the two actions is compatible with the inclusion $A[H]\subset A[G]$. We write $\Mod^{\rm pro\,aug}_G(A)$ for the category of profinite augmented $A$-representation of $G$, with morphisms being continuous $A[G]$-linear maps. The Pontrjagin dual $M\mapsto M^{\vee}:=\Hom^{\rm cont}_{\cO}(M,E/\cO)$ induces an anti-equivalence of categories
\begin{equation}\label{pondual}
\Mod^{\rm sm}_G(A)\xrightarrow{\rm anti\,\sim}\Mod_G^{\rm pro\, aug}(A).
\end{equation}
We say that $\tau\in \Modsm$ is locally finite if for all $v \in \tau$, the $\cO[G]$-submodule generated by $v$ is of finite length. We write $\Modlfin$ for the full subcategory of $\Modsm$ consisting of locally finite objects. We similarly define the category $\Modladm$ of locally admissible representations. It follows from \cite[Theorem 2.3.8]{EmertonOrd1} that $\Modladm=\Modlfin$. We let $\CO$ be the full subcategory of $\Modproaug$ which is anti-equivalent to $\Modladm$ under (\ref{pondual}). Following \cite[\S 2]{Paskunas-JL}, we define an exact covariant functor $\checkV\colon\CO\to\Mod^{\rm pro}_{G_{\Q_p}}(\cO)$ as follows. If $M\in\CO$ is of finite length, then we define $\checkV(M)=\V(M^{\vee})^{\vee}\otimes\z$. For general $M\in\CO$, write $M=\varprojlim M_i$ with $M_i$ of finite length, we define $\checkV(M)=\varprojlim\checkV(M_i)$. In this normalization, we have
\hfill\begin{itemize}
\item $\checkV((\chi\circ\det)^{\vee})=0$,
\item $\checkV((\Sp\otimes\chi\circ\det)^{\vee})=\chi\o^{-1}$,
\item $\checkV((\Ind_{B}^{G}\chi_1\otimes\chi_2\omega^{-1})^{\vee})=\chi_1\o^{-1}$,
\item $\checkV(\pi(r,0,\chi)^{\vee})\cong\V(\pi(r,0,\chi))\otimes\o^{-1}$.
\end{itemize}

Let $\Modfg$ be the category of profinite augmented representations of $G$ over $\cO$ with a central character $\z^{-1}$ whose underlying module is finitely generated over $\cO[\![H]\!]$ for some (equivalently any) compact open subgroup $H$. Let $\Pi$ be an admissible unitary Banach representation of $G$ over $E$ with central character $\z$. Let $\Theta\subset\Pi$ be a $G$-invariant open bounded $\cO$-lattice. Then $\Theta^{d}:=\Hom^{\rm cont}_{\cO}(\Theta,\cO)$ is an object in $\CO$. We define $\checkV(\Pi):=\checkV(\Theta^d)\otimes_{\cO} E$ so that $\checkV$ is an exact contravariant functor on $\Banadm$. 

We say a unitary irreducible admissible $E$-Banach space representation of $G$ is \textit{ordinary} if it's a subquotient of a unitary parabolic induction of a unitary character. We say $\Pi$ is \textit{non-ordinary} if it is not ordinary. The following theorem (see \cite[Theorem 11.4]{PaskunasIHES} and \cite[Theorem 1.1]{MR3272011}) is celebrated as the $p$-adic local Langlands correspondence.
\begin{theorem}
The functor $\checkV$ induces a bijection between isomorphism classes of
\hfill\begin{enumerate}
\item absolutely irreducible admissible unitary non-ordinary $E$-Banach space representations of $G$ with the central character $\zeta$, and
\item absolutely irreducible $2$-dimensional continuous $E$-representations of $G_{\Q_p}$
with determinant equal to $\z\e^{-1}$.
\end{enumerate}
\end{theorem}
Let $\rho$ be an absolutely irreducible $2$-dimensional continuous $E$-representations of $G_{\Q_p}$. We write $\Pi(\rho)$ for the corresponding absolutely irreducible admissible unitary non-ordinary $E$-Banach space representations of $G$ such that $\checkV(\Pi(\rho))\cong\rho$.

\subsection{The modulo $p$ correspondence}\label{modpcor}
Let $\o_2\colon I_{\Q_p}\to\F^{\times}$ be the fundamental character of level $2$ defined in \cite[\S 1.7]{serre1972proprietes} given by 
\begin{equation}\label{w_2}
\omega_2(g)=\frac{g((-p)^{\frac{1}{p^2-1}})}{(-p)^{\frac{1}{p^2-1}}},\; \forall g\in I_{\Q_p}.
\end{equation}
This definition does not depend on the choice of $(-p)^{\frac{1}{p^2-1}}$ and shows that $\o_2$ extends to a character $G_{\Q_{p^2}}\to \F^{\times}$. We say $h\in\Z/ (p^2-1)\Z $ is \textit{primitive} if $ph\not\equiv h\pmod {p^2-1}$. If $h$ is primitive, then $\Ind^{G_{\Q_p}}_{G_{\Q_{p^2}}}\o^h_2$ is irreducible. Conversely, every absolutely irreducible $2$-dimensional $\F$-linear representation of $G_{\Q_p}$ is isomorphic to 
\[\rho(h,\g):=(\Ind^{G_{\Q_p}}_{G_{\Q_{p^2}}}\o^h_2)\otimes\g\] 
for some $h\in\Z/ (p^2-1)\Z$ primitive and some unramified character $\g$. See alse \cite[Lemma 2.1.4]{berger2010some}. It is clear that $\rho(h,\g)\simeq\rho(h',\g')$ if and only if $\g^2=(\g')^2$, $h=h'$ or $\g^2=(\g')^2$, $ph=h'$. Since $\o_2^{p+1}=\o$, we have
\begin{equation}\label{irrgal}
(\Ind^{G_{\Q_p}}_{G_{\Q_{p^2}}}\o^h_2)\otimes\g\cong(\Ind^{G_{\Q_p}}_{G_{\Q_{p^2}}}\o_2^{r+1})\otimes\o^s\g
\end{equation}
for $h=s(p+1)+r+1$ with $0\le r\le p-1$ and $0\le s\le p-2$. 

Let $\overline{\rho}\colon G_{\Q_p}\to \GL_2(\F)$ be a continuous representation with $\det\overline{\rho}=\z\o^{-1}$ such that $\End_{G_{\Q_p}}(\overline{\rho})=\F$.  We define $\pi(\overline{\rho})\in\ModadmF$ such that $\checkV(\pi(\overline{\rho})^{\vee})\cong \overline{\rho}$ in the following manner. 
\begin{enumerate}
\item If $\overline{\rho}$ is absolutely irreducible, then $\pi(\overline{\rho})$ is the (unique) supersinguler representation such that  $\checkV(\pi(\overline{\rho})^{\vee})\cong \overline{\rho}$.
\item If  $\overline{\rho}\cong\big(\begin{smallmatrix}\chi_1&*\\0&\chi_2\end{smallmatrix}\big)$ with $\chi_1\chi_2^{-1}\neq\mathbf{1},\o^{\pm1}$, then $\pi(\overline{\rho})$ is a nonsplit extension 
\[
0\to\Ind_{B}^{G}\chi_2\o\otimes\chi_1\to\pi(\overline{\rho})\to\Ind_{B}^{G}\chi_1\o\otimes\chi_2\to0.
\]
\item If $\overline{\rho}\cong\big(\begin{smallmatrix}\chi\o^{-1}&*\\0&\chi\end{smallmatrix}\big)$, then $\pi(\overline{\rho})$ has a unique Jordan–H\"older filtration
\[
0\subset\pi_1\subset\pi_2\subset\pi(\overline{\rho})
\]
such that $\pi_1\cong\pi_{\a}\otimes(\chi\circ\det)$, $\pi_2/\pi_1\cong\Sp\otimes(\chi\circ\det)$ and $\pi(\overline{\rho})/\pi_2\cong(\chi\circ\det)^{\oplus2}$.
\item If $\overline{\rho}\cong\big(\begin{smallmatrix}\chi&*\\0&\chi\o^{-1}\end{smallmatrix}\big)$, then $\pi(\overline{\rho})\cong\beta\otimes\chi\circ\det$, where $\beta$ is the representation defined in \cite[Lemma 6.7]{Paskunas-BM}.
\end{enumerate}
The existence of $\pi(\overline{\rho})$ follows from the dimension results on extensions reviewed in \S \ref{block}.
\begin{definition}\label{generic}
We say a continuous absolutely irreducible representation $\overline{\rho}\colon G_{\Q_p}\to\GL_2(\F)$ is generic if $\overline{\rho}\cong\rho(h,\g)$ for $h\not\equiv1, 2, p-1, p\pmod{p+1}$ and $\g$ unramified. This is equivalent to requiring $2\le r\le p-3$ in (\ref{irrgal}). We say a smooth irreducible supersingular representation $\pi$ of $G$ generic if $\V(\pi)$ is generic, or equivalently, $\pi\cong\pi(r,0,\chi)$ for some smooth character $\chi$ and $2\le r\le p-3$.
\end{definition}

\section{Scholze's functor}\label{Scholze}
Let $D$ be the (unique) non-split quaternion algebra over $\Q_p$. To any $\pi\in\Mod^{\rm adm}_{G}(\cO)$, Scholze \cite{Scholze} constructs a Weil-equivariant sheaf $\cF_{\pi}$ on the $\et$ale site of the adic space $\bP_{\C_p}^{1}$. The cohomology groups
\begin{equation*}
\cS^i(\pi)=H_{\et}^i(\bP_{\C_p}^{1},\cF_{\pi})
\end{equation*}
carry a continuous $G_{\Q_p}\times D^{\times}$-action and are admissible smooth representations of $D^{\times}$. We collect some results on Scholze's functor $\{\cS^i\}_{i\ge0}$.
\begin{theorem}Let $\pi$ be an admissible smooth representation of $G$ over $\F$.
\hfill\begin{enumerate} 
\item We have $\cS^i(\pi)=0$ for $i>2$.
\item If $\pi$ carries a central character, then the center of $D^{\times}$ acts on $\cS^i(\pi)$ by the same character.
\item If $\pi$ is supersingular, then $\cS^0(\pi)=0$; If we further assume $\pi$ is generic (see Definition \ref{generic}), then  $\cS^2(\pi)=0$.
\item If $\pi\cong\Ind_{B}^{G}\chi_1\otimes\chi_2$ with $\chi_1\neq\chi_2$ or if $\pi\cong\Sp\otimes(\chi\circ\det)$, then $\cS^0(\pi)=\cS^2(\pi)=0$. 
\item We have $\cS^0(\mathbf{1}_G)\cong\mathbf{1}_{G_{\Q_p}}\otimes\mathbf{1}_{D^{\times}}$, $\cS^2(\mathbf{1}_G)=\o^{-1}\otimes\mathbf{1}_{D^{\times}}$ and $\cS^1(\mathbf{1}_G)=0$.
\end{enumerate}
\end{theorem}
\begin{proof}
(1) follows from \cite[Theorem 3.2]{Scholze} and (2) is proved in \cite[Lemma 7.3]{DPS-Crelle}. The vanishing of $\cS^0(\pi)$ in (3) and (4) is a direct corollary of \cite[Proposition 4.7]{Scholze}. The vanishing of $\cS^2(\pi)$ in (3) is \cite[Theorem 1.2]{HW-JL1}. The vanishing of $\cS^2(\pi)$ in (4) is proved in \cite[Theorem 4.6 and Corollary 4.7]{Ludwig}. As for (5), we we note that $\cF_{\mathbf{1}_G}$ is the trivial local system on $\bP^{1}_{\C_p}$ and by \cite[Theorem 3.8.1]{Huber} the cohomology of $\bP^{1}_{\C_p}$ (with the Galois action) is as in the classical case. Since $D^{\times}$ acts on $\bP^{1}_{\C_p}$ via an embedding $D^{\times} \to\GL_2(\Q_p^{\rm un})$, $D^{\times}$ acts trivially on the cohomology.
\end{proof}

For $\pi\in\Modladm$, we write $\pi=\varinjlim\pi'$ with $\pi'$ running though all the admissible subrepresentations of $\pi$. As in \cite[\S3.1]{Paskunas-JL}, we define $ \cS^i(\pi):=H_{\et}^i(\bP_{\C_p}^{1},\cF_{\pi})$ and there is an isomorphism
\[ \cS^i(\pi)\cong\varinjlim H_{\et}^i(\bP_{\C_p}^{1},\cF_{\pi'})\cong\varinjlim\cS^i(\pi').\] 
Thus $\cS^i(\pi)$ is a locally admissible smooth representation of $D^{\times}$. Define a covariant homological $\d$-functor $\{\Shat^i\}_{i\ge0}$ on $\CO$ by
\[
\Shat^i(M):=H_{\et}^i(\bP_{\C_p}^{1},\cF_{M^{\vee}})^{\vee}.
\]

Let $\Pi\in\Banadm$ and let $\Theta\subset \Pi$ be an open bounded $\cO$-lattice invariant under the $G$-action. We then define 
\[\Shat^i(\Pi):=\Shat^i(\Theta^d)^d\otimes_{\cO}E,\] 
which is an admissible unitary $E$-Banach space representation of $D^{\times}$ (\cite[Lemma 3.4]{Paskunas-JL}). Let ${\rm Ban}_{G_{\Q_p}\times D^{\times},\z}^{\rm adm}(E)$ be the category of unitary $E$-Banach space representations of $G_{\Q_p}\times D^{\times}$ which also belong to ${\rm Ban}_{D^{\times},\z}^{\rm adm}(E)$.
\begin{proposition}\label{deltafunctor}
The functors $\{\Shat^i\}_{i\ge0}$ define a cohomological $\d$-functor from the category ${\rm Ban}_{G,\z}^{\rm adm}(E)$ to the category  ${\rm Ban}_{G_{\Q_p}\times D^{\times},\z}^{\rm adm}(E)$ .
\end{proposition}
\begin{proof}
Let $\Modfg_{\rm tor}$  be the full subcategory of $\Modfg$ consisting of $\cO$-torsion modules. Then
\[
\begin{aligned}
\Modfg&\to{\rm Ban}_{G,\z}^{\rm adm}(E)\\
M&\mapsto \Hom^{\rm cont}_{\cO}(M,\cO)\otimes_{\cO}E
\end{aligned}
\]
induces an anti-equivalence between $\Modfg/\Modfg_{\rm tor}$ and ${\rm Ban}_{G,\z}^{\rm adm}(E)$. Similarly we have an anti-equivalence 
\[
\begin{aligned}
{\rm Mod}_{D^{\times},\z}^{\rm fg\,aug}(\cO)/{\rm Mod}_{D^{\times},\z}^{\rm fg\,aug}(\cO)_{\rm tor}&\xrightarrow{\sim}{\rm Ban}_{D^{\times},\z}^{\rm adm}(E).
\end{aligned}
\]
Since every object in  $\Modfg_{\rm tor}$ is killed by some power of $\varpi$, we have $M\in\Modfg_{\rm tor}$ implies $\Shat^i(M)\in{\rm Mod}_{D^{\times},\z}^{\rm fg\,aug}(\cO)_{\rm tor}$ for all $i\ge0$. we can deduce from the diagram
\[
\begin{tikzcd}
\Modfg \arrow[d, "\Shat^i"] \arrow[r]                    & \Modfg/\Modfg_{\rm tor} \arrow[r, "\sim"]                                                                               & {{\rm Ban}_{G,\z}^{\rm adm}(E)} \arrow[d, "\Shat^i", dashed] \\
{{\rm Mod}_{D^{\times},\z}^{\rm fg\,aug}(\cO)} \arrow[r] & {{\rm Mod}_{D^{\times},\z}^{\rm fg\,aug}(\cO)/{\rm Mod}_{D^{\times},\z}^{\rm fg\,aug}(\cO)_{\rm tor}} \arrow[r, "\sim"] & {{\rm Ban}_{D^{\times},\z}^{\rm adm}(E)}                    
\end{tikzcd}
\]
that $\{\Shat^i\}_{i\ge0}\colon {\rm Ban}_{G,\z}^{\rm adm}(E)\to{\rm Ban}_{D^{\times},\z}^{\rm adm}(E)$ is a cohomological $\d$-functor.
\end{proof}
\begin{lemma}\label{S^0}
 Suppose $\Pi\cong (\Ind^G_B\d_2\e\otimes\d_1)_{\rm cont}$ for unitary characters $\d_1,\d_2\colon\Q^{\times}_p\to E^{\times}$.
\hfill \begin{enumerate}
 \item If $\d_1/\d_2\neq\e$, then $\Shat^0(\Pi)=0$.
 \item if $\d_1/\d_2=\e$, then $\Shat^0(\Pi)\cong\d_{1}\otimes\d_1\circ\Nrd$ as a $G_{\Q_p}\times D^{\times}$-representation.
 \end{enumerate}
\end{lemma}
\begin{proof}
If $\d_1/\d_2\not\equiv\e\pmod\varpi$, then $\Ind^G_B\overline{\d_2}\e\otimes\overline{\d_1}$ is irreducible and $\cS^0(\Ind^G_B\overline{\d_2}\e\otimes\overline{\d_1})=0$. Let $\Theta\subset \Pi$ be an open bounded $\cO$-lattice invariant under the $G$-action. By\cite[(14)]{Paskunas-JL}, there is an exact sequence
\[
\Shat^0(\Theta^d)\xrightarrow{\varpi}\Shat^0(\Theta^d)\to\Shat^0(\Theta^d/\varpi)\to0.
\]
Since $\Shat^0(\Theta^d/\varpi)=0$, we have $\Shat^0(\Theta^d)=0$ by topological Nakayama’s lemma. So $\Shat^0(\Pi)=0$.

Suppose $\d_1/\d_2\equiv\e\pmod\varpi$ and $\d_1/\d_2\neq\e$. Let $n$ be an positive integer such that $\d_1/\d_2\equiv\e\pmod{\varpi^{n-1}}$ and $\d_1/\d_2\not\equiv\e\pmod{\varpi^n}$. Then for $m\ge n$ we have
\[
\cS^0(\Theta/\varpi^m)=\cS^0((\Theta/\varpi^m)^{\SL_2(\Q_p)})\cong\cS^0((\d/\varpi^{n-1})\circ\det),
\]
which implies $\Shat^0(\Theta^d)^d=0$ by \cite[Lemma 3.3]{Paskunas-JL}. So $\Shat^0(\Pi)=0$.

Suppose $\d_1/\d_2=\e$. Then for each integer $m\ge 0$ we have
\[
\cS^0(\Theta/\varpi^m)=\cS^0((\Theta/\varpi^m)^{\SL_2(\Q_p)})\cong\cS^0((\d_1/\varpi^m)\circ\det)\cong\d_1/\varpi^m\otimes\d_1/\varpi^m\circ\Nrd.
\]
By \cite[Lemma 3.3]{Paskunas-JL} $\Shat^0(\Theta^d)^d$ is a free $\cO$-module of rank $1$ on which $G_{\Q_p}\times D^{\times}$ acts by $\d_{1}\otimes\d_1\circ\Nrd$. So $\Shat^0(\Pi)\cong\d_{1}\otimes\d_1\circ\Nrd$.
\end{proof}
\section{Global arguments}\label{global}

\subsection{Quaternion algebras and modularity}\label{qf}
We fix a totally real field $F$ with $[F:\Q]$ even. There exists a definite quaternion algebra $B'$ over $F$, ramified exactly at all the infinite places of $F$. Let $\mathcal{O}_{B'}$ be a maximal order of $B'$. For each finite place $v$, we fix an isomorphism $(\mathcal{O}_{B'})_v\cong M_2(\mathcal{O}_{F_v})$ which extends to an isomorphism $(B'\otimes_FF_v)^{\times}\cong\GL_2(F_v)$.

Let $U\subset \prod_v\GL_2(F_v)$ be a compact open subgroup of $(B'\otimes_F\A^{\infty}_F)^{\times}$. We write $\Sigma_p$ for the set of finite places of $F$ lying over $p$. Fix a finite place $\p\in\Sigma_p$. We denote by $U^p:=\prod_{v\notin\Sigma_p }U_v$ and $U^{\p}:=\prod_{v\neq\p}U_v$. Assume that $U$ is small enough in the sense that (2.1.2) of \cite{Kisin-FM} holds, i.e. for all $t\in (B'\otimes_F\A^{\infty}_F)^{\times}$, we have 
\begin{equation}\label{small}
(U(\A_F^{\infty})^{\times}\cap t^{-1}(B')^{\times}t)/F^{\times}=1.
\end{equation}

Let $A$ be a topological $\cO$-algebra. We define $S(U^p,A)$ be the space of continuous functions 
\[
f\colon (B')^{\times}\backslash (B'\otimes_F\A^{\infty}_F)^{\times}/U^p\to A
\]

Let $\psi\colon(\A_{F}^{\infty})^{\times}/F^{\times}\to A^{\times}$ be a continuous character such that $\psi|_{U_v\cap F_v^{\times}}$ is trivial when $v\nmid p$. Define
\[
S_{\psi}(U^p,A):= S(U^p,A)[\psi].
\]

For each $v\mid p$, let $V_{\lambda_v}$ be a finite free $A$-module with a continuous action of $U_v$ and such that $F^{\times}_v\cap U_v$ acts by $\psi|_{F_v^{\times}}$. Denote by $V_{\lambda}=\otimes_{v\in\Sigma_p}V_{\lambda_v}$. We define 
\[
S_{\psi,\l}(U,A):=\Hom_{U_p}(V_{\l},S_{\psi}(U^p,A)).
\]

Let $S$ be the union of $\Sigma_{p}$ and the set of finite places $v$ such that $U_v\neq\GL_2(\cO_{F_v})$. Let $\T^{S,\mathrm{univ}}=\cO[T_v,S_v\colon v\notin S]$ then $\T^{S,\mathrm{univ}}$ acts on $S_{\psi,\l}(U,A)$ in the usual way, where $T_v$ and $S_v$ act via the double cosets
\[
\GL_2(\cO_{F_v})
\begin{pmatrix}
\varpi_v&0\\
0&1
\end{pmatrix}
\GL_2(\cO_{F_v})
\]
and
\[
\GL_2(\cO_{F_v})
\begin{pmatrix}
\varpi_v&0\\
0&\varpi_v
\end{pmatrix}
\GL_2(\cO_{F_v})
\] 
respectively. We will also consider Hecke operators 
\[
W_w=U_w
\begin{pmatrix}
\varpi_w&0\\
0&1
\end{pmatrix}
U_w
\]
at  $w\in S\backslash\Sigma_p$. 

Let $\overline{r}\colon G_F\to\GL_2(\F)$ be an absolutely irreducible continuous representation unramified outside $S$. We write $\fm_{\overline{r}}$ for the ideal of $\T^{S,\mathrm{univ}}$ generated by $\varpi$ and $T_v-\mathrm{tr}(\overline{r}(\mathrm{Frob}_v))$, $\mathbf{N}(v)S_v-\mathrm{det}(\overline{r}(\mathrm{Frob}_v))$ for all $v\notin S$. Here $\mathrm{Frob}_v\in G_{F}$ is a (geometric) Frobenius element.
\begin{definition}
Let $A=\F$. Suppose $U_v=\GL_2(\cO_{F_v})$ and $V_{\l_v}$ is irreducible for each $v\mid p$ . We say $\overline{r}$ is \textit{modular} of weight $\l$ if there exist some $U$ and $(\psi,\l)$ as above, such that $\overline{r}$ is unramified outside $S$ and $S_{\psi,\l}(U,\F)_{\fm_{\overline{r}}}\neq0$.
\end{definition}

\begin{remark}\label{Serreweights}The weights of a modular Galois representation $\overline{r}$ are predicted by the local Serre weights $W(\overline{r}|_{G_{F_v}})$ for $v\mid p$. This is the Buzzard--Diamond--Jarvis conjecture \cite[Conjecture 3.14]{BDJ} which is proved in \cite[Theorem B]{Gee-Kisin}. See also \cite[Theorem 3.17]{BDJ} for an explicit description of $W(\overline{r}|_{G_{F_v}})$ when $F_v\cong\Q_p$. 
\end{remark}
We write $\fm$ for $\fm_{\overline{r}}$. Define
\[
\begin{aligned}
S_{\psi}(U^{p},\cO)_{\fm}&:=\varprojlim_{s}\varinjlim_{U_{p}}S_{\psi}(U^{p}U_{p},\cO/\varpi^s)_{\fm}\\
S_{\psi}(U^{\p},\cO)_{\fm}&:=\varprojlim_{s}\varinjlim_{U_{\p}}S_{\psi}(U^{\p}U_{\p},\cO/\varpi^s)_{\fm}.
\end{aligned}\]
We denote by $\l^{\p}=\otimes_{v\mid p, v\neq\p}\l_{v}$. Define
\[
S_{\psi,\l^{\p}}(U^{\p},\cO)_{\fm}:=\Hom_{U_p^{\p}}(V_{\l_{\p}},S_{\psi}(U^p,\cO)_{\fm}).
\]
\begin{lemma}
If $\overline{r}$ is modular, then $S_{\psi}(U^{p},\F)_{\fm}\neq0$ for $\psi\e^{-1}\equiv\det\overline{r}\pmod\varpi$. Conversely if $S(U^{p},\F)_{\fm}\neq0$, then $\overline{r}$ is modular.
\end{lemma}
\begin{proof}
See \cite[Lemma 5.3]{Paskunas-JL}.
\end{proof}

\subsection{Completed cohomology}

Let $B$ be the indefinite quaternion algebra split at one infinite place, say $\infty_F$, ramified at $\frak{p}$, and having the same ramification as $B'$ at other places. We choose a maximal order $\cO_{B}$ of $B$ as well as isomorphisms $\cO_{B_v}\cong {\rm M}_2(\cO_{F_v})$ for $v\neq\p$. By abuse of notation, we write $U=U_{\p}U^{\p}$ for the compact open subgroup of $(B\otimes_F\A^{\infty}_F)^{\times}$ with $U_{\p}$ an open subgroup of $\cO_{B_{\p}}^{\times}$ and $U^{\p}\subset (B\otimes_F\A^{\infty, \p}_F)^{\times}$ the same as the definite case. There is a smooth projective algebraic curve $X_U$ associated to $U$ over $F$ with 
\begin{equation*}
X_U(\C)=B^{\times}\backslash\big((B\otimes_F\A^{\infty}_F)^{\times}/U\times(\bP^1(\C)\backslash\bP^1(\R))\big).
\end{equation*}
We define the completed cohomology groups
\begin{equation*}
\begin{aligned}
\widehat{H}^1(U^{\frak{p}},\cO)_{\fm}&:=\varprojlim_n\varinjlim_{U_{\frak{p}}}H^1_{\et}((X_{U^{\p}U_{\p}})_{\overline{F}},\cO/\varpi^n)_{\fm}\\
\widehat{H}^1(U^{p},\cO)_{\fm}&:=\varprojlim_n\varinjlim_{U_{p}}H^1_{\et}((X_{U^{p}U_{p}})_{\overline{F}},\cO/\varpi^n)_{\fm}.
\end{aligned}
\end{equation*}
Write $\widehat{H}^1_{\psi}(U^{p},\cO)_{\fm}$ for the maximal submodule of $\widehat{H}^1(U^{p},\cO)_{\fm}$ on which $(\A^{\infty}_F)^{\times}$ acts by the character $\psi$.
we also define
\[
\widehat{H}^1_{\psi,\lambda^{\p}}(U^{\frak{p}},\cO)_{\fm}:=\Hom_{U_p^{\p}}(V_{\l^{\p}},\widehat{H}^1_{\psi}(U^{p},\cO)_{\fm}),
\]
We denote by $\T(U^{p}U_{p})$ the image of $\T^{S,\mathrm{univ}}$ in $\End_{\cO}(H^1_{\et}((X_{U^{p}U_{p}})_{\overline{F}},\cO))$ and write $\T(U^{p}U_{p})_{\fm}$ for the $\fm$-adic completion of $\T(U^{p}U_{p})$. Then
\[
\T(U^{p})_{\fm}:=\varprojlim_{U_{p}}\T(U^{p}U_{p})_{\fm}
\]
acts faithfully on $\widehat{H}^1(U^{p},\cO)_{\fm}$. Let $G_{F,S}$ be the Galois group of the maximal extension of F unramified outside $S$.
\begin{proposition}
There is a unique (up to conjugation) continuous $2$-dimensional Galois representation 
\[r_{\fm}\colon G_{F,S}\to\GL_2(\T(U^p)_{\fm})\]
 unramified outside $S$, such that for all $v\notin S$,
 \[
 {\rm tr}(r_{\fm}({\rm Frob}_v)) = T_v, \; \det(r_{\fm}({\rm Frob}_v)) = \mathbf{N}(v)S_v.
 \]
 The ring $\T(U^{p})_{\fm}$ is a complete noetherian local ring with finite residue field. The $\T(U^{p})_{\fm}[G_{F,S}]$-module $\widehat{H}^1(U^{p},\cO)_{\fm}$ is $r_{\fm}$-typic (see Definition 5.2 of \cite{Scholze}).
\end{proposition}
\begin{proof}
Similar as Proposition 5.7 and Proposition 5.8 of \cite{Scholze}
\end{proof}
Let $\chi\colon G_F\to\cO^{\times}$ be a continuous character such that $\chi\equiv1\pmod\varpi$. Assume $\chi|_{U_v\cap F_v^{\times}}$ is trivial when $v\nmid p$. We also view $\chi$ as a continuous character $\chi\colon(\A_F)^{\times}/ F^{\times} \to\cO^{\times}$ via the global Artin map $\phi\colon (\A_F)^{\times}/ F^{\times} \to G_F^{\rm ab}$.
\begin{lemma}\label{completedcohomologytwist}
There is a $(B\otimes_{\Q} \Q_p)^{\times}\times G_F$-equivariant isomorphism
\[
\begin{aligned}
\a\colon\widehat{H}^1_{\psi}(U^{p},\cO)_{\fm}\otimes(\chi\circ\det\boxtimes\chi)\cong\widehat{H}^1_{\psi\chi^2}(U^{p},\cO)_{\fm}
\end{aligned}
\]
such that for each finite place $v\notin S$, 
\[
T_v\circ\a=\chi(\varpi_v)\a \circ T_v,\; S_v\circ\a=\chi(\varpi_v)^2\a\circ S_v.
\]
\end{lemma}
\begin{proof}
We denote by $F_U:=H^0(X_U,\cO_{X_U})$, which is a finite abelian extension of $F$ (see \cite[\S1.2]{MR703228}). Let $\cF_{\chi\circ\det/\varpi^n}$ be the sheaf on $X_U$ associated to the representation $U\xrightarrow{v} (\A_F^{\infty})^{\times}\xrightarrow{\chi}\cO/\varpi^n$. Using the same argument as \cite[Lemma 2.3]{BDJ}, one can show
 \begin{equation}\label{H=ind}
 H^0((X_U)_{\overline{F}},\cF_{\chi\circ\det/\varpi^n})\cong \Ind_{G_{F_U}}^{G_F}(\chi/\varpi^n)
 \end{equation}
 as $G_F$-representations. Let $v\colon (B\otimes_F\A_F^{\infty})^{\times}\to (\A_F^{\infty})^{\times}$ be the reduced norm. For $g\in (B\otimes_{F}A^{\infty}_F)^{\times}$, the diagram 
\[
\begin{tikzcd}
{H^0((X_U)_{\overline{F}},\cF_{\chi\circ\det/\varpi^n})} \arrow[r] \arrow[d, "g^*"] &  \Ind_{G_{F_U}}^{G_F}(\chi/\varpi^n) \arrow[d,"\phi(v(g))^*"] \\
{H^0((X_{gUg^{-1}})_{\overline{F}},\cF_{\chi\circ\det/\varpi^n})} \arrow[r]  &  \Ind_{G_{F_{U}}}^{G_F}(\chi/\varpi^n) 
\end{tikzcd}
\]
commutes. Let $\mathbf{1}_{n,U}\in\Ind_{G_{F_U}}^{G_F}(\cO/\varpi^n)$ be the constant function with value $1\in\cO/\varpi^n$. Suppose $\a_{n,U_p}\in H^0((X_{U_pU^p})_{\overline{F}},\cF_{\chi\circ\det/\varpi^n})$ corresponds to 
\[
\chi\otimes\mathbf{1}\in\chi\otimes (\Ind_{G_{F_U}}^{G_F}\cO/\varpi^n)\cong(\Ind_{G_{F_U}}^{G_F}\chi/\varpi^n)
\] 
under isomorphism (\ref{H=ind}). Then cupping with $\{\a_{n,U_p}\}$ defines the desired isomorphism.\end{proof}

\subsection{Globalization}\label{globalization}
We start with a continuous representation $\overline{\rho}\colon G_{\Q_p}\to\GL_2(\F)$. By Proposition 8.1 of  \cite{DPS-Crelle}, there exist a totally real field $F$ and a regular algebraic cuspidal automorphic weight $0$ representation $\pi$ of $\GL_2(\A_F)$ such that the associated Galois representation $r\colon G_F\to \GL_2(\overline{\Q}_{p})$ satisfies:
\hfill\begin{enumerate}
\item $r$ is unramified outside $\Sigma_p$;
\item $p$ splits completely in $F$ and $\overline{r}|_{G_{F_v}}\cong\overline{\rho}$ for all $v\mid p$;
\item $\SL_2(\F)\subset \overline{r}(G_F)\subset\GL_2(\F)$;
\item $[F:\Q]$ is even.
\end{enumerate}

Let $U_{\rm max}=\prod_{v\nmid \infty}\GL_2(\cO_{F_v})$. Let $N$ be the product of the orders of the groups $(U_{\rm max}(\A_F^{\infty})^{\times}\cap t_i^{-1}(B')^{\times}t)/F^{\times}$, where $t_i$ runs through a (finite) set of  representatives of the quotient $(B')^{\times}\backslash(B'\otimes_F\A^{\infty}_F)^{\times}/(\A_F^{\infty})^{\times}U_{\rm max}$. According to \cite[Proposition 8.2]{DPS-Crelle}, there is a finite place $w_1$ of $F$ with the following properties:
\hfill\begin{enumerate}
\item $\mathbf{N}w_1\not\equiv 1\mod p$;
\item the ratio of the eigenvalues of $\overline{r}(\mathrm{Frob}_{w_1})$ is not equal $(\mathbf{N}w_1)^{\pm1}$ or $1$;
\item $\mathbf{N}w_1$ is prime to $2pN$.
\end{enumerate}

Let $U_{w_1}$ be the subgroup of $\GL_2(\cO_{F_{w_1}})$ consisting of elements that are upper-triangular and unipotent modulo $\varpi_{w_1}$ and let $U_v=\GL_2(\cO_{F_v})$ for $v\neq w_1$. Then by \cite[Lemma 3.2]{pavskunas20162}, $U=\prod_{v\nmid\infty}U_v$ is small enough such that for all $t\in (B'\otimes_F\A^{\infty}_F)^{\times}$, (\ref{small}) holds.

\begin{lemma}
The Galois representation $\overline{r}\colon G_F\to \GL_2(\F)$ is modular, i.e., 
\[
S(U^p,\F)_{\fm}\neq0.
\]
\end{lemma}
\begin{proof}
See Proposition 8.4 of \cite{DPS-Crelle}.
\end{proof}
Let $\psi\colon G_F\to \cO^{\times}$ the character such that $\psi\e^{-1}=\det r$. We also view $\psi$ as a continuous character $\psi\colon(\A_F^{\infty})^{\times}/ F^{\times} \to\cO^{\times}$ via the global Artin map.

There exists an irreducible $E$-representation $\sigma_{v}$ of $\GL_2(\F_p)$ such that its modulo $\varpi$ reduction contains one of $W(\overline{r}|_{G_{F_v}})$ (see Remark \ref{Serreweights}) as a subquotient, see \cite[Lemma 4.4.1]{gee2011automorphic} for a precise choice. Let $\sigma^0_{v}$ be a $\GL_2(\F_p)$-invariant lattice of $\sigma_{v}$. Then $\sigma^0_{v}$ has a central character $\z_{v}\colon \Z_p^{\times}\to\cO^{\times}$ and clearly $\z_{v}\equiv \psi|_{\cO_{F_v}^{\times}}\pmod \varpi$. There exists a character $\theta_{v}\colon \Z_p^{\times}\to\cO^{\times}$ such that $\z_{v}\theta_{v}^2= \psi|_{\cO_{F_v}^{\times}}$. We take $\l_{v}=\sigma^0_{v}\otimes(\theta_{v}\circ\det)$ and $\l^{\p}=\otimes_{v\mid p, v\neq\p}\l_{v}$. Then $S_{\psi,\l^{\p}}(U^{\p},\cO)_{\fm}\neq0$ by our choice of $(\psi,\l)$. 

\subsection{Galois deformation rings}

We write $\overline{r}_v$ for the restriction of $\overline{r}$ to $G_{F_v}$ and write $\psi_{v}$ for $\psi|_{G_{F_v}}$. Let $R_v^{\Box,\psi\e^{-1}}$ be the universal framed deformation ring corresponding to liftings of $\overline{r}_v$ with determinant $\psi_{v}\e^{-1}$. 
 Let $R^{\Box,\psi\e^{-1}}_{\Sigma_p}=\widehat{\otimes}_{\cO}R_{v}^{\Box,\psi\e^{-1}}$ for $v$ running over elements in $\Sigma_p$. 
 
Let $\sigma$ be a finite dimensional irreducible $E$-representation of $K:=\GL_2(\Z_p)$ with central character $\z|_{\Z_p^{\times}}$. We say $\s$ is a locally algebraic type if 
\[\sigma\simeq \Sym^bE^2\otimes{\rm det}^a\otimes\sigma(\tau)\]
 for some $a\in\Z$, $b\in\Z_{\ge0}$ and for some inertial type $\tau\colon I_{\Q_p}\to\GL_2(E)$. Here $\sigma(\tau)$ is the smooth irreducible representation of $\GL_2(\Z_p)$ attached to $\tau$ by the inertial local Langlands correspondence. For $v\in\Sigma_p$, we also write $R_{v}^{\Box,\psi\e^{-1}}(\sigma)$ for the maximal $\varpi$-torsion free reduced quotient of $R_{v}^{\Box,\psi\e^{-1}}$ parametrizing potentially semi-stable liftings of $\overline{r}_v$ with Hodge--Tate weights $(1-a,-a-b)$ and having inertial type $\tau$.
  
 Let $\s'$ be a finite dimensional irreducible $E$-representation of $\cO_D^{\times}$. We say $\s'$ is a locally algebraic type if 
\[\sigma'\simeq \Sym^bE^2\otimes{\rm Nrd}^a\otimes\sigma'(\tau)\]
  for some $a\in\Z$, $b\in\Z_{\ge0}$ and for some discrete series inertial type $\tau\colon I_{\Q_p}\to\GL_2(E)$. Here $\sigma'(\tau)$ (up to a conjugation of $\varpi_D$) is the smooth irreducible representation of $\cO_D^{\times}$ attached to $\tau$ by \cite[Theorem 3.3]{Gee-Geraghty}. We also write $R_{v}^{\Box,\psi\e^{-1}}(\sigma')$ for the maximal $\varpi$-torsion free reduced quotient of $R_{v}^{\Box,\psi\e^{-1}}$ parametrizing potentially semi-stable liftings of $\overline{r}_v$ with Hodge--Tate weights $(1-a,-a-b)$ and having inertial type $\tau$.

  Let $Q$ be a set of finite places of $F$ disjoint from $S$. We write $S_{Q}=S\cup Q$. Let $G_{F,S_{Q}}$ be the Galois group of the maximal extension of F unramified outside $S_{Q}$. Let $R^{\Box,\psi\e^{-1}}_{F,Q}$ be the universal $\Sigma_p$-framed deformation ring of $\overline{r}:G_{F,S_{Q}}\to \GL_2(\F)$ with fixed determinant $\psi\e^{-1}$, i.e., $R^{\Box,\psi\e^{-1}}_{F,Q}$ represents the functor assigning to a local Artinian $\cO$-algebra $A$ the set of isomorphism classes of tuples $\{V_A, \b_v\}_{v\in\Sigma_p}$, where $V_A$ is a deformation of $V_{\F}$ to $A$ with determinant $\psi\e^{-1}$ and
$\b_v$ is a lifting of the chosen basis of $V_{\F}$ to an $A$-basis of $V_A$. See also \cite[Proposition 2.2.9]{MR2470687}. 
  
There is a canonical $\cO$-algebra homomorphism
 \[
 R^{\Box,\psi\e^{-1}}_{\Sigma_p}\to R^{\Box,\psi\e^{-1}}_{F,Q}
  \]
by sending a tuple $\{V_A, \b_v\}_{v\in\Sigma_p}$ to $(V_A,\b_v)$, for each $v\in\Sigma_p$.
\begin{proposition}\label{TWprimes}
Let $r:=\dim_{\F} H^1(G_{F,S}, ({\rm ad}^0\overline{r})(1))$. For each positive integer $N$, there exists a finite set $Q_N$ of finite primes of $F$ satisfying
\begin{enumerate}
\item $Q_N$ is disjoint from $S$ and $|Q_N|\equiv r$.
\item  If $v\in Q_N$, then $\mathbf{N}v\equiv1\pmod{p^N}$.
\item If $v\in Q_N$, then $\overline{r}({\rm Frob}_v)$ has distinct eigenvalues.
\item Set $g=r-[F\colon Q]+|\Sigma_p|-1$, then $R^{\Box,\psi\e{-1}}_{F,Q_N}$ is topologically generated over $R^{\Box,\psi\e^{-1}}_{\Sigma_p}$ by $g$ elements. In particular $g\ge0$.
\end{enumerate}
\end{proposition}
\begin{proof}
All the conditions of \cite[(3.2.3)]{MR2600871} are satisfied by our choices of $\overline{r}$ and $S$, so we can refer to \cite[Proposition (3.2.5)]{MR2600871}.
\end{proof}
\subsection{Patching}If  $ v\notin \Sigma_p\cup Q_N$ is a finite place of $F$, let $U(N)_v: =U_v$.
If $v\in Q_N$,
we let
\[
U(N)_v=\big\{\begin{pmatrix}a&b\\c&d\end{pmatrix}\in\GL_2(\cO_{F_v})| c\equiv0\pmod w,\,  ad^{-1}\mapsto 1\in \Delta(N)_v\big\},
\]
where $\Delta(N)_v$ is the maximal $p$-power quotient of $k_v^{\times}$. Define
\[U(N)^{p}:=\prod_{v\notin \Sigma_p} U(N)_v\subset (B'\otimes_F\A_F^{p,\infty})^{\times}.\]

The Hecke algebra $\T^{\rm univ}_{S_{Q_N},\cO}=\cO[T_v, S_v, W_w]_{v\not\in S_{Q_N}, w\in Q_N}$ acts on $H^1(X_{U(N)^{p}U_{p}}\times_F\overline{F},\cO)$ for all open compact subgroups $U_p\subset(B\otimes_{\Q}\Q_{p})^{\times}$ in the usual way. We denote by $\T(U(N)^{p}U_{p})$ the image of $\T^{\rm univ}_{S_{Q_N},\cO}$ in $\End_{\cO}\big(H^1(X_{U(N)^{p}U_{p}}\times_F\overline{F},\cO)\big)$. For each $w\in Q_N$, we choose an eigenvalue $\alpha_w$ of $\overline{r}(\mathrm{Frob}_w)$. Define 
\[\fm_{Q_N}:=(\varpi, T_v-\mathrm{tr}(\overline{r}(\mathrm{Frob}_v)),\mathbf{N}(v)S_v-\mathrm{det}(\overline{r}(\mathrm{Frob}_v)), U_{\pi_w}-\alpha_w)_{v\not\in S_{Q_N}, w\in Q_N },
\]
which is a maximal ideal of $\T^{\rm univ}_{S_{Q_N},\cO}$. We define
\[
\begin{aligned}
S_{\psi}(U(N)^{p},\cO)_{\fm_{Q_N}}&:=\varprojlim_{s}\varinjlim_{U_{p}}S_{\psi}(U(N)^{p}U_{p},\cO/\varpi^s)_{\fm_{Q_N}},\\
S_{\psi,\l^{\p}}(U(N)^{\p},\cO)_{\fm_{Q_N}}&:=\Hom_{U_p^{\p}}((V_{\l_{\p}})^d,S_{\psi}(U(N)^p,\cO)_{\fm_{Q_N}}),\\
\widehat{H}^1(U(N)^{p},\cO)_{\fm_{Q_N}}&:=\varprojlim_n\varinjlim_{U_{p}}H^1_{\et}((X_{U(N)^{p}U_{p}})_{\overline{F}},\cO/\varpi^n)_{\fm_{Q_N}},\\
\widehat{H}^1_{\psi,\lambda^{\p}}(U(N)^{\frak{p}},\cO)_{\fm_{Q_N}}&:=\Hom_{U_p^{\p}}((V_{\l^{\p}})^d,\widehat{H}^1_{\psi}(U(N)^{p},\cO)_{\fm_{Q_N}}),\\
\T(U(N)^{p})_{\fm_{Q_N}}&:=\varprojlim_{U_{p}}\T(U(N)^{p}U_{p})_{\fm_{Q_N}}.
\end{aligned}
\]
\begin{proposition}
There is a unique (up to conjugation) continuous $2$-dimensional Galois representation 
\[r_{\fm_{Q_N}}\colon G_{F,S_{Q_N}}\to\GL_2(\T(U(N)^p)_{\fm_{Q_N}})\]
 unramified outside $S_{Q_N}$, such that for all $v\notin S_{Q_N}$,
 \[
 {\rm tr}(r_{\fm_{Q_N}}({\rm Frob}_v)) = T_v, \; \det(r_{\fm_{Q_N}}({\rm Frob}_v)) = \mathbf{N}(v)S_v.
 \]
 The ring $\T(U(N)^{p})_{\fm_{Q_N}}$ is a complete noetherian local ring with finite residue field. The $\T(U(N)^{p})_{\fm_{Q_N}}[G_{F,S_{Q_N}}]$-module $\widehat{H}^1(U(N)^{p},\cO)_{\fm_{Q_N}}$ is $r_{\fm_{Q_N}}$-typic (see Definition 5.2 of \cite{Scholze}).
\end{proposition}
\begin{proof}
Similar as Proposition 5.7 and Proposition 5.8 of \cite{Scholze}
\end{proof}
Let $R_{F,Q_N}$ be the universal deformation ring of $\overline{r}\colon G_{F,S_{Q_N}}\to \GL_2(\F)$. Then $r_{\fm_{Q_N}}$ induces a continuous $\cO$-algebra homomorphism $R_{F,Q_N}\to\T(U(N)^p)_{\fm_{Q_N}}$ which makes $\widehat{H}^1(U(N)^{p},\cO)_{\fm_{Q_N}}$ an $R_{F,Q_N}[G_{F,S_{Q_N}}]$-module. Let $R^{\psi\e^{-1}}_{F,Q_N}$ be the universal deformation ring of $\overline{r}:G_{F,S_{Q_N}}\to \GL_2(\F)$ with fixed determinant $\psi\e^{-1}$ and let $V_{Q_N}$ be the universal deformation. We denote by
\[r_{Q_N}\colon G_{F,S_{Q_N}}\to \GL_2(V_{Q_N})\]
the corresponding representation. Then the $R_{F,Q_N}$-action on $\widehat{H}^1_{\psi,\lambda^{\p}}(U(N)^{\frak{p}},\cO)_{\fm_{Q_N}}$ factors through $R^{\psi\e^{-1}}_{F,Q_N}$. We define
\begin{equation*}
\begin{aligned}
M(N)&:=\big(S_{\psi,\l^{\p}}(U(N)^{\p},\cO)_{\fm_{Q_N}}\big)^d \widehat{\otimes}_{{R^{\psi\e^{-1}}_{F,Q_N}}}{R^{\Box,\psi\e^{-1}}_{F,Q_N}}\\
L(N)&:=\big(\widehat{H}^1_{\psi,\lambda^{\p}}(U(N)^{\frak{p}},\cO)_{\fm_{Q_N}}\big)^d\widehat{\otimes}_{{R^{\psi\e^{-1}}_{F,Q_N}}}{R^{\Box,\psi\e^{-1}}_{F,Q_N}}\\
L'(N)&:=\big(\widehat{H}^1_{\psi,\lambda^{\p}}(U(N)^{\frak{p}},\cO)_{\fm_{Q_N}}[r_{\fm_{Q_N}}]\big)^d\widehat{\otimes}_{{R^{\psi\e^{-1}}_{F,Q_N}}}{R^{\Box,\psi\e^{-1}}_{F,Q_N}},
\end{aligned}
\end{equation*}
where
\[
\widehat{H}^1_{\psi,\lambda^{\p}}(U(N)^{\frak{p}},\cO)_{\fm_{Q_N}}[r_{\fm_{Q_N}}]:=\Hom_{\T(U(N)^p)_{\fm_{Q_N}}[G_{F,S_{Q_N}}]}(r_{\fm_{Q_N}},\widehat{H}^1_{\psi,\lambda^{\p}}(U(N)^{\frak{p}},\cO)_{\fm_{Q_N}}).
\]
By Proposition \ref{TWprimes}, we can and do fix a surjection 
\begin{equation}\label{localsujglobal}
R^{\Box,\psi\e^{-1}}_{\Sigma_p}[\![x_1,\ldots,x_g]\!]\twoheadrightarrow R^{\Box,\psi\e^{-1}}_{F,Q_N}.
\end{equation}

We define a quotient of $R^{\Box,\psi\e^{-1}}_{\Sigma_p}$ by 
\begin{equation*}
R^{\rm loc}:= R_{\p}^{\Box,\psi\e^{-1}}\widehat{\otimes}_{\cO}(\widehat{\otimes}_{v\mid p,v\neq\p}R_v^{\Box,\psi\e^{-1}}(\l_{v})).
\end{equation*}
Denote by $R_{\infty}=R^{\rm loc}[\![x_1,\ldots,x_g]\!]$. Let $\rho_{\p}^{\Box}\colon G_{F_{\p}}\to \GL_2(R_{\p}^{\Box,\psi\e^{-1}})$ be the universal lifting of $\overline{r}_{\p}$. Let $V^{\Box}_{\p}$ be an $R_{\p}^{\Box,\psi\e^{-1}}$-module free of rank $2$ with a basis on which $G_{F_{\p}}$ acts via  $\rho_{\p}^{\Box}$. 
\begin{lemma}\label{Htypic}
The $R^{\Box,\psi\e^{-1}}_{\Sigma_p}[\![x_1,\ldots,x_g]\!]$-actions on $M(N)$, $L(N)$ and $L'(N)$ via (\ref{localsujglobal}) factor through $R_{\infty}$. And there is a canonical (up to scalar) $R_{\infty}[G_{F_{\p}}\times B_{\p}^{\times}]$-equivariant isomorphism 
 \[
 L(N)\cong (\rho_{\p}^{\Box})^*\footnote{If $M\in{\rm Mod}_G(R)$ is finite free over $R$, we write $M^*$ for $\Hom_R(M,R)\in{\rm Mod}_G(R)$, which is also free of the same rank.}\boxtimes_{R^{\Box,\psi\e^{-1}}_{\p}}L'(N).
 \]
\end{lemma}
\begin{proof}
The first statement follows from the local-global compatibility. It remains to prove the second. 
Since both $r_{\fm_{Q_N}}$ and $r_{Q_N}$ are free of rank $2$, there is an isomorphism of $R^{\psi\e^{-1}}_{F,Q_N}$-modules 
\[
\widehat{H}^1_{\psi,\lambda^{\p}}(U(N)^{\frak{p}},\cO)_{\fm_{Q_N}}[r_{\fm_{Q_N}}]\cong\widehat{H}^1_{\psi,\lambda^{\p}}(U(N)^{\frak{p}},\cO)_{\fm_{Q_N}}[r_{Q_N}].
\]
Note that $\widehat{H}^1_{\psi,\lambda^{\p}}(U(N)^{\frak{p}},\cO)_{\fm_{Q_N}}$ is also $r_{Q_N}$-typic, i.e., we have
\[
\widehat{H}^1_{\psi,\lambda^{\p}}(U(N)^{\frak{p}},\cO)_{\fm_{Q_N}}\cong V_{Q_N}\otimes_{R^{\psi\e^{-1}}_{F,Q_N}}\widehat{H}^1_{\psi,\lambda^{\p}}(U(N)^{\frak{p}},\cO)_{\fm_{Q_N}}[V_{Q_N}].
\]
Therefore we have
\[
\widehat{H}^1_{\psi,\lambda^{\p}}(U(N)^{\frak{p}},\cO)_{\fm_{Q_N}}^d\cong (V_{Q_N})^{*}\otimes_{R^{\psi\e^{-1}}_{F,Q_N}}\big(\widehat{H}^1_{\psi,\lambda^{\p}}(U(N)^{\frak{p}},\cO)_{\fm_{Q_N}}[V_{Q_N}]\big)^d,
\]
and
\[
\begin{aligned}
L(N)&\cong (V_{Q_N})^{*}\otimes_{R^{\psi\e^{-1}}_{F,Q_N}}\big(\widehat{H}^1_{\psi,\lambda^{\p}}(U(N)^{\frak{p}},\cO)_{\fm_{Q_N}}[V_{Q_N}]\big)^d\widehat{\otimes}_{{R^{\psi\e^{-1}}_{F,Q_N}}}{R^{\Box,\psi\e^{-1}}_{F,Q_N}}\\
&\cong (V_{Q_N})^{*}\otimes_{R^{\psi\e^{-1}}_{F,Q_N}}\big(\widehat{H}^1_{\psi,\lambda^{\p}}(U(N)^{\frak{p}},\cO)_{\fm_{Q_N}}[r_{\fm_{Q_N}}]\big)^d\widehat{\otimes}_{{R^{\psi\e^{-1}}_{F,Q_N}}}{R^{\Box,\psi\e^{-1}}_{F,Q_N}}\\
&\cong (V_{Q_N})^{*}\otimes_{R^{\psi\e^{-1}}_{F,Q_N}}L'(N)\\
&\cong \big((V_{Q_N})^{*}\otimes_{R^{\psi\e^{-1}}_{F,Q_N}}R^{\Box,\psi\e^{-1}}_{F,Q_N}\big)\otimes_{R^{\Box,\psi\e^{-1}}_{F,Q_N}}L'(N)\\
&\cong \big(V_{Q_N}\otimes_{R^{\psi\e^{-1}}_{F,Q_N}}R^{\Box,\psi\e^{-1}}_{F,Q_N}\big)^{*}\otimes_{R^{\Box,\psi\e^{-1}}_{F,Q_N}}L'(N).
\end{aligned}
\]
We denote by $V^{\Box}_{Q_N}:=V_{Q_N}\otimes_{R^{\psi\e^{-1}}_{F,Q_N}}R^{\Box,\psi\e^{-1}}_{F,Q_N}$. Then as the underlying representation of the universal $\Sigma_p$-framed deformation, $V^{\Box}_{Q_N}$ has a canonical basis $\b_{\p}$ (up to scalar). From the definition of the $\cO$-algebra homomorphism $R^{\Box,\psi\e^{-1}}_{\p}\to R^{\Box,\psi\e^{-1}}_{F,Q_N}$, there is a unique canonical isomorphism
\[
V^{\Box}_{\p}\otimes_{R^{\Box,\psi\e^{-1}}_{\p}}R^{\Box,\psi\e^{-1}}_{F,Q_N}\cong V^{\Box}_{Q_N}
\]
identifying the corresponding bases. Thus
\[
\begin{aligned}
L(N)&\cong \big(r_{Q_N}\otimes_{R^{\psi\e^{-1}}_{F,Q_N}}R^{\Box,\psi\e^{-1}}_{F,Q_N}\big)^{*}\otimes_{R^{\Box,\psi\e^{-1}}_{F,Q_N}}L'(N)\\
&=(\rho_{\p}^{\Box}\otimes_{R^{\Box,\psi\e^{-1}}_{\p}}R^{\Box,\psi\e^{-1}}_{F,Q_N})^*\otimes_{R^{\Box,\psi\e^{-1}}_{F,Q_N}}L'(N)\\
&=(\rho_{\p}^{\Box})^*\otimes_{R^{\Box,\psi\e^{-1}}_{\p}}L'(N).
\end{aligned}
\]
\end{proof}
Set $j:=4|\Sigma_p|-1$ and define $\mathcal{J}=\cO[\![x_1,\ldots,x_{j}]\!]$. Since $R^{\Box,\psi\e^{-1}}_{F,Q_N}$ is formally smooth over $R^{\psi\e^{-1}}_{F,Q_N}$ of relative dimension $j$, we can and do fix an isomorphism
\begin{equation}\label{RtensorJ}
R^{\psi\e^{-1}}_{F,Q_N}\widehat{\otimes}_{\cO}\mathcal{J}\cong R^{\Box,\psi\e^{-1}}_{F,Q_N}.
\end{equation}
 Recall that for $v\in Q_N$, the finite cyclic group $\Delta(N)_v$ is the maximal $p$-power quotient of $k_v^{\times}$. We denote by $\Delta(N):=\prod_{v\in Q_N}\Delta(N)_v$.  Choose a generator for each $\Delta(N)_w$ then we have a surjection $\cO[\![y_1,\ldots,y_r]\!]\twoheadrightarrow\cO[\Delta(N)]$, whose kernel is the ideal $\mathcal{I}_N=((y_1+1)^{p^{t_1(N)}}-1,\ldots,(y_r+1)^{p^{t_r(N)}}-1)$ with $t_1(N),\ldots,t_r(N)\ge N$. There is group homomorphism $\D(N)\to (R^{\psi\e^{-1}}_{F,Q_N})^{\times}$ (see \cite[Lemma 2.1]{MR2290604}), which extends to 
 \begin{equation}\label{gpaction}
 \cO[\![y_1,\ldots,y_r]\!]\widehat{\otimes}_{\cO}\mathcal{J}\to R^{\psi\e^{-1}}_{F,Q_N}\widehat{\otimes}_{\cO}\mathcal{J}\xrightarrow{(\ref{RtensorJ})}R^{\Box,\psi\e^{-1}}_{F,Q_N}.
 \end{equation}
Denote by \[S_{\infty}:= \cO[\![y_1,\ldots,y_r]\!]\widehat{\otimes}_{\cO}\mathcal{J}.\]
We can view $M(N)$, $L(N)$ and $L'(N)$ as $S_{\infty}$-modules via (\ref{gpaction}).

For each open ideal $\mathfrak{a}$ of $S_{\infty}$, Let $I_{\fa}$ be set of positive integers $N$ such that $\mathcal{I}_{\fa}\subset \fa$. Then $\mathcal{I}_{\fa}$ is cofinite in $\Z_{\ge1}$. Fix a nonprincipal ultrafilter $\mathscr{F}$ on the set $\Z_{\ge1}$, which defines a point $x \in \Spec(\prod_{I_{\fa}}S_{\infty}/\fa)$ by  \cite[Lemma 2.2.2]{MR4386819}. Hence we have the map $\prod_{I_{\fa}}S_{\infty}/\fa\to S_{\infty}/\fa$, which is the localization map of $\prod_{I_{\fa}}S_{\infty}/\fa$ at $x$.  For each open compact subgroup $U_{\p}\subset B^{\times}_{\p}$, let
\begin{equation*}
 \begin{aligned}
  M(U_{\p},\fa,\infty)&:= \big(\prod_{N\in I_{\fa}}(M(N)/\fa)_{U_{\p}}\big)\otimes_{(\prod_{I_{\fa}}S_{\infty}/\fa)}S_{\infty}/\fa\\
 L(U_{\p},\fa,\infty)&:= \big(\prod_{N\in I_{\fa}}(L(N)/\fa)_{U_{\p}}\big)\otimes_{(\prod_{I_{\fa}}S_{\infty}/\fa)}S_{\infty}/\fa\\
 L'(U_{\p},\fa,\infty)&:=\big(\prod_{N\in I_{\fa}}(L'(N)/\fa)_{U_{\p}}\big)\otimes_{(\prod_{I_{\fa}}S_{\infty}/\fa)}S_{\infty}/\fa.
 \end{aligned}
 \end{equation*}
 We define 
 \[
 M_{\infty}:=\varprojlim_{U_{\p},\fa}M(U_{\p},\fa,\infty)
 \]
 and
 \[
 L_{\infty}:=\varprojlim_{U_{\p},\fa}L(U_{\p},\fa,\infty),\quad L'_{\infty}:=\varprojlim_{U_{\p},\fa}L'(U_{\p},\fa,\infty).
 \]
 \begin{proposition}\label{L=2L'}
There is an $R_{\infty}[G_{F_{\p}}\times B_{\p}^{\times}]$-equivariant isomorphism 
  \[
 L_{\infty}\cong (\rho_{\p}^{\Box})^*\boxtimes_{R^{\Box,\psi\e^{-1}}_{\p}}L_{\infty}'.
 \]
 \end{proposition}
 \begin{proof}
 By Lemma \ref{Htypic}, we pick an isomorphism
  \[
 L(N)\cong (\rho_{\p}^{\Box})^*\boxtimes_{R^{\Box,\psi\e^{-1}}_{\p}}L'(N)
 \]
 for each $N\ge1$. Therefore
 \begin{equation*}
 \begin{aligned}
 L(U_{\p},\fa,\infty)&= \big(\prod_{N\in I_{\fa}}(L(N)/\fa)_{U_{\p}}\big)\otimes_{(\prod_{I_{\fa}}S_{\infty}/\fa)}S_{\infty}/\fa\\
&\cong\big(\prod_{N\in I_{\fa}}\rho^{\Box}_{\p}\otimes_{R^{\Box,\psi\e^{-1}}_{\p}}(L'(N)/\fa)_{U_{\p}}\big)\otimes_{(\prod_{I_{\fa}}S_{\infty}/\fa)}S_{\infty}/\fa\\
 \end{aligned}
 \end{equation*}
 As $(\rho^{\Box}_{\p})^*$ is free of rank $2$, we have
 \begin{equation*}
 \begin{aligned}
 &\big(\prod_{N\in I_{\fa}}\rho^{\Box}_{\p}\otimes_{R^{\Box,\psi\e^{-1}}_{\p}}(L'(N)/\fa)_{U_{\p}}\big)\otimes_{(\prod_{I_{\fa}}S_{\infty}/\fa)}S_{\infty}/\fa\\
&\cong(\rho^{\Box}_{\p})^*\otimes_{R^{\Box,\psi\e^{-1}}_{\p}}\big(\prod_{N\in I_{\fa}}(L'(N)/\fa)_{U_{\p}}\big)\otimes_{(\prod_{I_{\fa}}S_{\infty}/\fa)}S_{\infty}/\fa.\\
&\cong(\rho^{\Box}_{\p})^*\otimes_{R^{\Box,\psi\e^{-1}}_{\p}}L'(U_{\p},\fa,\infty).
 \end{aligned}
 \end{equation*}
 Again by the same reason,
 \[
\begin{aligned}
L_{\infty}&=\varprojlim_{U_{\p},\fa}L(U_{\p},\fa,\infty)\\
&\cong\varprojlim_{U_{\p},\fa}(\rho^{\Box}_{\p})^*\otimes_{R^{\Box,\psi\e^{-1}}_{\p}}L'(U_{\p},\fa,\infty)\\
&\cong(\rho^{\Box}_{\p})^*\otimes_{R^{\Box,\psi\e^{-1}}_{\p}} \varprojlim_{U_{\p},\fa}L'(U_{\p},\fa,\infty)\\
&=(\rho^{\Box}_{\p})^*\otimes_{R^{\Box,\psi\e^{-1}}_{\p}}L'_{\infty}.
\end{aligned}
\]
Since $S_{\infty}$ acts on $M(N)$, $L(N)$ and $L'(N)$ via (\ref{gpaction}), all the isomorphisms are $R_{\infty}$-equivariant.
 \end{proof}
 \begin{remark}
Since everything is framed, we don't need the assumption \[\End_{G_{F_{\p}}}(\overline{r}|_{G_{F_{\p}}})=\F.\]
But the isomorphism in Proposition \ref{L=2L'} is not unique (even up to scalar).
\end{remark}
 Let $\chi\colon(\A^{\infty}_F)^{\times}/ F^{\times} \to1+\varpi\cO$ be a continuous character such that $\chi|_{U_v\cap F_v^{\times}}$ is trivial when $v\nmid p$. Recall that twisting by $\chi|_{F^{\times}_{\p}}$ (which we also write $\chi_{\p}$) defines an isomorphism
\[
{\rm tw}_{\chi}\colon R_{\p}^{\Box,\psi\chi^2}\xrightarrow{\sim} R_{\p}^{\Box,\psi}.
\]
\begin{proposition} \label{L_infty twist}
Let $\sigma'$ be a locally algebraic type of $B_{\p}^{\times}$, and let $(\sigma'\otimes\chi_{\p}\circ\Nrd)^0$  be a $\cO_{B_{\p}}^{\times}$-invariant $\cO$-lattice of $\sigma'\otimes\chi_{\p}\circ\Nrd$. Then the action of $R^{\Box,\psi\e^{-1}}_{\p}$ on
\[
L_{\infty}((\sigma'\otimes\chi_{\p}\circ\Nrd)^0):=\Hom_{\cO[\![\cO_{B_{\p}}^{\times}]\!]}^{\rm cont}(L_{\infty},((\sigma'\otimes\chi_{\p}\circ\Nrd)^0)^d)^d
\]
factors through the quotient $R^{\Box,\psi\e^{-1}}_{\p}(\sigma'\otimes\chi_{\p}\circ\Nrd)$, which is defined by the following diagram
\begin{equation*}
\begin{tikzcd}
{R^{\Box,\chi^{-2}\psi\e^{-1}}_{\p}} \arrow[d,two heads] \arrow[r, "{\rm tw}_{\chi^{-1}}"] & {R^{\Box,\psi\e^{-1}}_{\p}} \arrow[d,two heads]                         \\
{R^{\Box,\chi^{-2}\psi\e^{-1}}_{\p}(\sigma')} \arrow[r]                          & {R^{\Box,\psi\e^{-1}}_{\p}(\sigma'\otimes\chi_{\p}\circ\Nrd)}.
\end{tikzcd}
\end{equation*}
\end{proposition}
 \begin{proof}
Using the same argument as the proof of the first part of \cite[Lemma 4.17]{CEGGPS1}, we reduce to prove that, for any $N\ge1$ and any $U_{p}^{\p}$ small enough, the same result holds for
\[
\Hom_{\cO^{\times}_{B_{\p}}}\big((\sigma'\otimes\chi_{\p}\circ\Nrd)^0,\widehat{H}^1_{\psi}(U(N)^{p}U_p^{\p},\cO)_{\fm_{Q_N}}\big)\otimes_{{R^{\psi\e^{-1}}_{F,Q_N}}}(R^{\Box,\psi\e^{-1}}_{F,Q_N})^{\vee}.
\]
By Lemma \ref{completedcohomologytwist}, we only need to prove the action of $R^{\Box,\chi^{-2}\psi\e^{-1}}_{\p}$ on
\begin{equation}\label{localglobal}
\Hom_{\cO^{\times}_{B_{\p}}}\big((\sigma')^0,\widehat{H}^1_{\psi\chi^{-2}}(U(N)^{p}U_p^{\p},\cO)_{\fm_{Q_N}}\big)\otimes_{{R^{\chi^{-2}\psi\e^{-1}}_{F,Q_N}}}(R^{\Box,\chi^{-2}\psi\e^{-1}}_{F,Q_N})^{\vee}
\end{equation}
factors through the quotient $R^{\Box,\chi^{-2}\psi\e^{-1}}_{\p}(\sigma')$. Since $R^{\Box,\chi^{-2}\psi\e^{-1}}_{F,Q_N}$ acts on (\ref{localglobal}) through an $\cO$-torsion free reduced quotient, it's enough to show that if $x$ is a maximal ideal of $R^{\Box,\chi^{-2}\psi\e^{-1}}_{\p}[1/p]$ and belongs to the support of (\ref{localglobal}), then the pull back of $x$ along the morphism \[R^{\Box,\chi^{-2}\psi\e^{-1}}_{\p}[1/p]\to R^{\Box,\chi^{-2}\psi\e^{-1}}_{F,Q_N}[1/p]\]belongs to $\Spec R^{\Box,\chi^{-2}\psi\e^{-1}}_{\p}(\sigma')[1/p]$. But this is a consequence of the local-global compatibility.
\end{proof}
\begin{proposition} \label{M_infty twist}
Let $\sigma$ be a locally algebraic type of $(B'_{\p})^{\times}$, and let $(\sigma\otimes\chi_{\p}\circ\det)^0$  be a $\cO_{B'_{\p}}^{\times}$-invariant $\cO$-lattice of $\sigma\otimes\chi_{\p}\circ\det$. Then the action of $R^{\Box,\psi\e^{-1}}_{\p}$ on
\[
M_{\infty}((\sigma\otimes\chi_{\p}\circ\det)^0):=\Hom_{\cO[\![\cO_{B'_{\p}}^{\times}]\!]}^{\rm cont}(M_{\infty},((\sigma\otimes\chi_{\p}\circ\det)^0)^d)^d
\]
factors through the quotient $R^{\Box,\psi\e^{-1}}_{\p}(\sigma\otimes\chi_{\p}\circ\det)$.
\end{proposition}
\begin{proof}
Use \cite[Proposition 5.5]{Paskunas-JL} and the same argument as the proof of Proposition \ref{L_infty twist}.
\end{proof}
\begin{corollary}\label{anycharacter}
Suppose $\d\colon F^{\times}_{\p}\to  \cO^{\times}$ is a continuous character such that $\d\equiv1\pmod\varpi$ and $\d^2$ is locally algebraic, then Proposition \ref{L_infty twist} and Proposition \ref{M_infty twist} remain true when $\chi_{\p}$ is replaced by $\d$.
\end{corollary}
\begin{proof}
Let $\tilde{\e}$ be the Teichmüller lifting of the modulo $p$ cyclotomic character. We write $\eta:=\sqrt{\tilde{\e}^{-1}\e}|_{G_{F_{\p}}}$. Since $\d$ is locally algebraic, we have $\d=\eta^a\d'$ for some smooth character $\d'$ and some integer $a$. Since the statement is true when $\d$ is smooth, we can reduce to the case when $\d=\eta^a$ for some integer $a$. But this is exactly Proposition \ref{L_infty twist} and Proposition \ref{M_infty twist} since $\eta$ is the restriction of $\sqrt{\tilde{\e}^{-1}\e}$.
\end{proof}
\subsection{Arithmetic action} \begin{definition}\label{arithmetic}
Let $R$ be a complete local noetherian ring which is faithfully flat over $R_{\p}^{\Box}$. Suppose $M\in\Mod^{\rm fg.aug}_{G}(R)$ is nonzero. We say the action of $R$ on $M$ is arithmetic (with respect to $R_{\p}^{\Box}$) if the following conditions hold:
\hfill\begin{enumerate}
\item\label{arith1} M is projective in $\Mod^{\rm pro}_{K}(\cO)$.
\item \label{arith2}For any locally algebraic type $\sigma$, we let $\sigma^0$  be a $\GL_2(\Z_p)$-invariant $\cO$-lattice of $\sigma$.   Then the action of $R_{\infty}$ over
\[
M(\sigma^0):=\Hom_{\cO[\![\GL_2(\Z_p)]\!]}^{\rm cont}(M,(\sigma^0)^d)^d
\]
factors through the quotient $R(\sigma):=R\otimes_{R_{\p}^{\Box}}R_{\p}^{\Box}(\sigma)$. Moreover, $M(\sigma^0)$ is finitely generated maximal Cohen-Macaulay over the $R(\s)$.
\item\label{arith3} For any $\sigma$, the action of $\mathcal{H}(\sigma)$ on 
\[
M(\sigma^{0})[1/p]\simeq \Hom_K(\s,M^d[1/p])'\footnote{$\Hom_K(\s,M^d[1/p])'$ is the weak dual of the Banach space $ \Hom_K(\s,M^d[1/p])$.}\simeq \Hom_G(\cInd_{K}^G\s,M^d[1/p])'
\]
coincides with the action given by the composite
\[
\cH(\s)\xrightarrow{\eta} R_{\p}^{\Box}(\sigma)[1/p]\to R(\s)[1/p],
\]
where $\eta\colon\cH(\s)\to R_{\p}^{\Box}(\sigma)[1/p]$ is defined in \cite[Theorem 4.1]{CEGGPS1}.
\end{enumerate}
\end{definition}
Let $\L$ be the universal deformation ring of the trivial $1$-dimensional representation of $G_{F_{\p}}$ and 
let $\mathbf{1}^{\rm univ}$ be the universal deformation. We also view $\mathbf{1}^{\rm univ}$ as a representation of $F^{\times}_{\p}$. There is an isomorphism 
\[
R^{\Box}_{\p}\cong R^{\Box,\psi\e^{-1}}_{\p}\widehat{\otimes}\L
\]
which makes $R_{\infty}\widehat{\otimes}_{\cO}\L$ an $R^{\Box}_{\p}$-algebra. Let $\sigma$ be a locally algebraic type of $(B'_{\p})^{\times}$. If $R^{\Box}_{\p}(\sigma)$ is nonzero, there exists a character $\d\colon F^{\times}_{\p}\to\cO^{\times}$ with trivial modulo $p$ reduction, such that $\psi\d^{-2}|_{\cO^{\times}_{F_{\p}}}$ is the central character of $\sigma$. Let $\L^{\rm ur}$ be the quotient of $\L$ corresponding to unramified characters. We have 
\[
R^{\Box}_{\p}(\sigma)\cong R^{\Box,\d^{-2}\psi\e^{-1}}_{\p}(\sigma)\widehat{\otimes}\L^{\rm ur}.
\]
And the diagram 
\[
\begin{tikzcd}
R^{\Box}_{\p} \arrow[d, two heads] \arrow[r, "\sim"]                      & {R^{\Box,\psi\e^{-1}}_{\p}\widehat{\otimes}\L} \arrow[d, two heads,"{\rm tw}_{\d}\otimes{\rm tw}_{\d^{-1}}"]    \\
R^{\Box}_{\p}(\sigma) \arrow[r, "\sim"] & {R^{\Box,\d^{-2}\psi\e^{-1}}_{\p}(\sigma)\widehat{\otimes}\L^{\rm ur}}
\end{tikzcd}
\]
is commutative.

\begin{lemma}\label{M twisti arith}
The action of $R_{\infty}\widehat{\otimes}_{\cO}\L$ on $M_{\infty}\widehat{\otimes}_{\cO} (\mathbf{1}^{\rm univ})^{-1}\circ\det$ is arithmetic with respect to $R_{\p}^{\Box}$.
\end{lemma}
\begin{proof}
For (\ref{arith1}) of Definition \ref{arithmetic}, see \cite[Proposition 6.10]{CEGGPS2}. For (\ref{arith2}), suppose $M(\sigma^0)$ is nonzero, then there exists a character $\d\colon F^{\times}_{\p}\to\cO^{\times}$ with trivial modulo $p$ reduction, such that $\psi\d^{-2}|_{\cO^{\times}_{F_{\p}}}$ is the central character of $\sigma$. According to \cite[Proposition 6.12]{CEGGPS2}, there is a natural isomorphism of $R_{\p}^{\Box}$-modules
\[
\Hom_{\cO[\![K]\!]}^{\rm cont}(M_{\infty}\widehat{\otimes}_{\cO}(\mathbf{1}^{\rm univ})^{-1}\circ\det,(\sigma^0)^d)^d\cong\Hom_{\cO[\![K]\!]}^{\rm cont}(M_{\infty}\otimes_{\cO}(\d\circ\det),(\sigma^0)^d)^d\widehat{\otimes}_{\cO}\L^{\rm ur},
\]
where $\L$ acts on the right-hand side by the morphism 
\[\L\xrightarrow{\rm tw_{\d^{-1}}}\L\to\L^{\rm ur}. \]
So it's enough to show the $R^{\Box,\psi\e^{-1}}_{\p}$-action on 
\[
\Hom_{\cO[\![K]\!]}^{\rm cont}(M_{\infty}\otimes_{\cO}(\d\circ\det),(\sigma^0)^d)^d
\]
factors through $R^{\Box,\psi\e^{-1}}_{\p}(\sigma\otimes\d\circ\det)$. Now we can apply Corollary \ref{anycharacter}. The statement (\ref{arith3}) of Definition \ref{arithmetic} can be prove using argument of the same style as (\ref{arith2}), which we omit.
\end{proof}
\begin{corollary}
Let $x\colon R_{\infty}\to \overline{\Q}_p$ be a continuous $\cO$-algebra morphism, and let $\fm_x$ be its kernel. Define \[\Pi_{M,x}:=\Hom_{\cO}^{\rm cont}(M_{\infty},E)[\fm_x].\] Then $\Pi_{M,x}$ is nonzero.
\end{corollary}
\begin{proof}
Suppose $y\colon\L\to\cO^{\times}$ corresponds to the trivial character. We define $x':=x\otimes y$. Then
\[
M_{\infty}\widehat{\otimes}_{\cO} (\mathbf{1}^{\rm univ})^{-1}\circ\det\otimes_{R_{\infty}\widehat{\otimes}_{\cO}\L,x'}\kappa(x')
\]
is non-zero by Lemma \ref{M twisti arith} and \cite[Theorem 6.5]{MR4077579}. So
\[
M_{\infty}\otimes_{R_{\infty},x}\kappa(x)
\]
is non-zero, which proves the corollary.
\end{proof}
\begin{proposition}\label{cocleofPi_x}
 Let $\rho$ be the Galois representation corresponding to the composition $R^{\Box,\psi\e^{-1}}_{\p}\to R_{\infty}\xrightarrow{x} \overline{\Q}_p$. 
\hfill \begin{enumerate}
 \item If $\rho$ is absolutely irreducible, then $\Pi_{M,x}\cong\Pi(\rho)^{\oplus n_x}$ for some $n_x\ge1$. 
 \item If  $\rho\cong\big(\begin{smallmatrix}\d_1&*\\0&\d_2\end{smallmatrix}\big)$ for unitary characters $\d_1,\d_2\colon\Q_p^{\times}\to E^{\times}$, then all the irreducible subquotients of $\Pi_{M,x}$ are the irreducible subquotients of 
\[
(\Ind^G_B\d_1\e\otimes\d_2)_{\rm cont}\oplus (\Ind^G_B\d_2\e\otimes\d_1)_{\rm cont}.
\]
\item Assume $\End_{G_{F_{\p}}}(\overline{r}|_{G_{F_{\p}}})=\F$. Suppose $\rho\cong\big(\begin{smallmatrix}\d_1&*\\0&\d_2\end{smallmatrix}\big)$ with $\d_1\d_2^{-1}\not\equiv 1$. If $\Pi$ is an irreducible closed subrepresentation of $\Pi_{M,x}$, then $\Pi$ is infinite dimensional and is a subquotient of $(\Ind^G_B\d_2\e\otimes\d_1)_{\rm cont}$.
\end{enumerate}
\end{proposition}
\begin{proof}
The first two statements are \cite[Theorem 7.1]{pavskunas2021finiteness}. Let $\fm_x$ be the kernel of $x$. Then 
\[\Theta:=\Hom^{\rm cont}_{\cO}(M_{\infty}/\fm_x,\cO)\]
 is a $G$-invariant lattice in $\Pi_{M,x}$. Let $\fm_{\infty}$ be the maximal ideal of $R_{\infty}$ and let 
 \[\pi^M(\overline{\rho}):=(M_{\infty}/\fm_{\infty})^{\vee}.
 \]
  We have isomorphisms
\[\Theta/\varpi\cong ((M_{\infty}/\fm_x)_{\rm tf}/\varpi)^{\vee}\text{ and }\:\pi^M(\overline{\rho})\cong ((M_{\infty}/\fm_x)/\varpi)^{\vee},
\]
which induce an embedding $\iota\colon\Theta/\varpi\hookrightarrow\pi^M(\overline{\rho})$. If $\Pi\subset\Pi_{M,x}$ is an irreducible subrepresentation, then $\Theta_{\Pi}:=\Theta\cap \Pi$ is an open bounded $G$-invariant $\cO$-lattice in $\Pi$.
We have a $G$-equivariant embedding $\Theta_{\Pi}/\varpi\hookrightarrow\Theta/\varpi\hookrightarrow\pi^M(\overline{\rho})$. It follows from \cite[Theorem 7.7]{HW-JL1} that $\pi^M(\overline{\rho})\cong\pi(\overline{\rho})^{\oplus d}$ for some $d\ge1$. Therefore $\Theta_{\Pi}/\varpi$ is a subrepresentation of $\pi(\overline{\rho})$, and this proves the last two statements.
\end{proof}
\begin{theorem}\label{S(M)=L}
There is an $R_{\infty}[B^{\times}_{\p}\times G_{F_{\p}}]$-equivariant isomorphism 
\[\Shat^1(M_{\infty})\cong L_{\infty}.\] Let $x\colon R_{\infty}\to \overline{\Q}_p$ be a continuous $\cO$-algebra morphism, and let $\fm_x$ be its kernel. Define 
\[\Pi_{L,x}:=\Hom_{\cO}^{\rm cont}(L_{\infty},E)[\fm_x].\] Then $\check{\mathcal{S}}^1(\Pi_{M,x})$ is a closed subrepresentation of $\Pi_{L,x}$ and $\Pi_{L,x}/\check{\mathcal{S}}^1(\Pi_{M,x})$ is a finite dimensional $E$-space on which the subgroup of reduced norm $1$ elements in $\cO^{\times}_{B_{\p}}$ acts trivially.
\end{theorem}
\begin{proof}
This follows from Theorem 8.10 and Lemma 8.14 of  \cite{DPS-Crelle}.
\end{proof}
\begin{corollary}\label{S(Pi)typic}
Suppose $\Pi$ is an absolutely  irreducible non-ordinary unitary Banach representation of $\GL_2(\Q_p)$.  Let $\rho$ be the Galois representation corresponding to it. Then $\Shat^1(\Pi)$ is $\rho$-typic, i.e., there exists a unitary Banach representation ${\rm JL}(\rho)$ of $D^{\times}$ and a $G_{\Q_p}\times D^{\times}$-isomorphism
\[
\Shat^1(\Pi)\cong \rho\boxtimes {\rm JL}(\rho).
\]
\end{corollary}
\begin{proof}
Choose a $G_{\Q_p}$-invariant $\cO$-lattice $\L$ of $\rho$, and let $\overline{\rho}\colon G_{\Q_p}\to\End_{\F}(\L/\varpi)$ be the modulo $\varpi$ reduction of $\overline{\rho}\colon G_{\Q_p}\to\End_{\cO}(\L)$. We patch the completed cohomologies from this $\overline{\rho}$. Since $\Shat^i$ and the $p$-adic Langlands correspondence are compatible with taking twist (see \cite[Lemma 7.4]{DPS-Crelle}), we can assume $\det\rho=\psi_{\p}\e^{-1}$. Suppose that $y\colon R_{\p}^{\Box,\psi\e^{-1}}\to E$ corresponds to $\rho$. Since $R_{\infty}$ is faithfully flat over $R_{\p}^{\Box,\psi\e^{-1}}$, there exists some $x\colon R_{\infty}\to \overline{\Q}_p$ extending $y$. It follows from Proposition \ref{L=2L'} that 
\[\Pi_{L,x}\cong\rho\boxtimes \Pi_{L',x}.
\]
So $\Shat^1(\Pi)$ is also $\rho$-typic because it's a sub-representation of $\Pi_{L,x}$ (See \cite[Proposition 5.4]{Scholze}).
\end{proof}
\begin{corollary}\label{O_Dalg}
Suppose $\Pi$ is an absolutely  irreducible unitary Banach representation of $\GL_2(\Q_p)$.
Let $\sigma'=\Sym^b E^2\otimes\Nrd^a\otimes \sigma'_{\rm sm}$ be a locally algebraic type of $\cO_D^{\times}$. Suppose
\begin{equation}\label{algneq0}\Hom_{\cO^{\times}_{D}}(\sigma',\Shat^1(\Pi))\neq0.\end{equation}
Let $\tau$ be the inertial type corresponding to $\sigma'_{\rm sm}$ (c.f. \cite[Theorem 3.3]{Gee-Geraghty}).
\hfill\begin{enumerate}
\item[(a)]  If $\Pi$ is non-ordinary, let $\rho$ be the Galois representation corresponding to it. Then $\rho$ is potentially semi-stable of Hodge--Tate weight $(1-a,-a-b)$ and has inertial type $\tau$.
\item[(b)] If $\Pi\cong (\Ind^G_B\d_2\e\otimes\d_1)_{\rm cont}$ for unitary characters $\d_1,\d_2\colon\Q^{\times}_p\to E^{\times}$ with $\d_1/\d_2\neq\e^{\pm1}$, we assume further $\d_1/\d_2\not\equiv1\pmod\varpi$. Let $\rho\cong\big(\begin{smallmatrix}\d_1&*\\0&\d_2\end{smallmatrix}\big)$ be the (unique) nonsplit extension of $\d_2$ by $\d_1$. Then $\rho$ is potentially semi-stable of Hodge--Tate weight $(1-a,-a-b)$ and has inertial type $\tau$.
\item[(c)] If $\Pi\cong\d\circ\det\otimes \widehat{\rm St}$, where $\widehat{\rm St}$ is the universal unitary completion of the smooth Steinberg representation of $G$, we let $\rho\cong\big(\begin{smallmatrix}\d&*\\0&\d\e^{-1}\end{smallmatrix}\big)$ be a nonsplit extension of $\d\e^{-1}$ by $\d$. Then $\rho$ is potentially semi-stable of Hodge--Tate weight $(1-a,-a-b)$ and has inertial type $\tau$.
\item[(d)]  If $\Pi\cong (\Ind^G_B\d\e\otimes\d\e^{-1})_{\rm cont}$ for a unitary character $\d$,  we let  $\rho\cong\big(\begin{smallmatrix}\d\e^{-1}&*\\0&\d\end{smallmatrix}\big)$ be a nonsplit extension of $\d$ by $\d\e^{-1}$. Then $b=0$, $\sigma'_{\rm sm}\cong \chi\circ\Nrd$ for some smooth character $\chi$ and $\d=\e^a\chi$.
\end{enumerate}
\end{corollary}
\begin{proof}
Choose a $G_{\Q_p}$-invariant $\cO$-lattice $\L$ of $\rho$, and let $\overline{\rho}\colon G_{\Q_p}\to\End_{\F}(\L/\varpi)$ be the modulo $\varpi$ reduction of $\rho\colon G_{\Q_p}\to\End_{\cO}(\L)$. Note that we can choose $\L$ such that $\overline{\rho}$ is indecomposable. We patch the completed cohomologies from this $\overline{\rho}$. By Corollary \ref{anycharacter}, we can assume $\det\rho=\psi_{\p}\e^{-1}$. Suppose that $y\colon R_{\p}^{\Box,\psi\e^{-1}}\to E$ corresponds to $\rho$. Since $R_{\infty}$ is faithfully flat over $R_{\p}^{\Box,\psi\e^{-1}}$, there exists some $x\colon R_{\infty}\to \overline{\Q}_p$ extending $y$. Then $\Pi$ is a sub-representation of $\Pi_{M,x}$ by Proposition \ref{cocleofPi_x}. By Proposition \ref{deltafunctor} we have an exact sequence 
\[
0\to\Shat^0(\Pi_{M,x}/\Pi)/\Shat^0(\Pi_{M,x})\to\Shat^1(\Pi)\to\Shat^1(\Pi_{M,x}).
\]
We claim that  $\Shat^0(\Pi_{M,x}/\Pi)/\Shat^0(\Pi_{M,x})\neq0$ only when $\Pi\cong (\Ind^G_B\d\e\otimes\d\e^{-1})_{\rm cont}$ or $\Pi\cong\d\circ\det\otimes \widehat{\rm St}$. To see this, if $\Pi$ is non-ordinary, then $\Pi_{M,x}\cong \Pi^{\oplus n_x}$ for some $n_x\ge1$. So the exact 
\[
0\to \Pi\to \Pi_{M,x}\to \Pi_{M,x}/\Pi\to0
\]
splits. So $\Shat^0(\Pi_{M,x}/\Pi)/\Shat^0(\Pi_{M,x})=0$. If $\Pi$ is ordinary,  then all the irreducible subquotients of $\Pi_{M,x}$ are the irreducible subquotients of 
\[
(\Ind^G_B\d_1\e\otimes\d_2)_{\rm cont}\oplus (\Ind^G_B\d_2\e\otimes\d_1)_{\rm cont}.
\]
By Lemma \ref{S^0}, $\Shat^0(\Pi_{M,x}/\Pi)/\Shat^0(\Pi_{M,x})$ is nonzero only when $\Pi\cong (\Ind^G_B\d\e\otimes\d\e^{-1})_{\rm cont}$ or $\Pi\cong\d\circ\det\otimes \widehat{\rm St}$. And in these cases, all the subquotients of $\Shat^0(\Pi_{M,x}/\Pi)/\Shat^0(\Pi_{M,x})$ are isomorphic to $\d\circ\Nrd$ as $D^{\times}$-representations.

If
\begin{equation}\label{S0alg}
\Hom_{\cO^{\times}_{D}}(\Sym^b E^2\otimes\Nrd^a\otimes \sigma'_{\rm sm},\Shat^0(\Pi_{M,x}/\Pi)/\Shat^0(\Pi_{M,x}))\neq0,
\end{equation}
then $b=0$ and $\Nrd^a\otimes\sigma'_{\rm sm}=\d\circ\Nrd$, i.e., $\sigma'_{\rm sm}\cong \chi\circ\Nrd$ for some smooth character $\chi$ and $\d=\e^a\chi$.

If the left-hand side of  (\ref{S0alg}) is zero, we have 
\begin{equation}
\Hom_{\cO^{\times}_{D}}(\Sym^b E^2\otimes\Nrd^a\otimes \sigma'_{\rm sm},\Shat^1(\Pi_{M,x}))\neq0,
\end{equation}
 Therefore according to Theorem \ref{S(M)=L} we have
\[\Hom_{\cO^{\times}_{D}}(\Sym^b E^2\otimes\Nrd^a\otimes \sigma'_{\rm sm},\Pi_{L,x})\neq0.\]
Combined with \cite[Proposition 2.22]{Paskunas-BM} we conclude that \[L_{\infty}(\Sym^b E^2\otimes\Nrd^a\otimes \sigma'_{\rm sm})^0)\otimes_{R_{\infty,x}}\overline{\Q}_p\neq0.\] Hence by Theorem \ref{L_infty twist}, $\fm_y$ is in $\Spec R_{\p}^{\Box,\psi\e^{-1}}(\s')$.
\end{proof}

\subsection{Locally algebraic vectors}
Let $\tilde{\e}$ be the Teichmüller lifting of the modulo $p$ cyclotomic character.  From now on, we define $\eta\colon \Q_p^{\times}\to E^{\times}$ to be the character $\sqrt{\tilde{\e}^{-1}\e}$.

 Let $D^{\times,1}$ be the subgroup of $D^{\times}$ consisting of elements with reduced norm equal to $1$. Let $\Pi$ be a unitary Banach representation of $D^{\times}$. We write $\Pi^{\onealg}$ for the subspace of $D^{\times,1}$-locally algebraic vectors in $\Pi$.
\begin{lemma}\label{findim}
Let $\Pi$ be a unitary Banach representation of $D^{\times}$ with a locally algebraic central character. Then $\Pi^{\onealg}$ is a direct sum of finite-dimensional irreducible $D^{\times}$-representations of the form $\Sym^b E^2 \otimes\Nrd^a \otimes V\otimes\eta^i\circ\Nrd$ for some $b\in\Z_{\ge0}$, $a\in\Z$, $i\in\{0,1\}$ and irreducible smooth representation $V$ of $D^{\times}$. Conversely, any finite-dimensional sub-representation of $\Pi$ is locally algebraic for the $D^{\times,1}$-action. 

\end{lemma}
\begin{proof}
This is \cite[Proposition 6.13]{Paskunas-JL}. For completeness, we sketch the proof. Since $D^{\times,1}$ is compact, $\Pi^{\onealg}$ is a direct sum of finite dimensional irreducible locally algebraic representations by \cite[Corollary 4.2.9]{emerton2004locally}. Suppose $V'$ is an irreducible smooth representation of $D^{\times,1}$ and that $W=\Sym^b E^2\otimes V'$ is a sub-$D^{\times,1}$-representation of $\Pi$. Since $\Pi$ has a locally algebraic central character, say $\z=\chi\e^c$ with $\chi$ smooth, $W$ is also $\Q_p^{\times}$-invariant. We extend $V'$ to a smooth representation $V$ of $D^{\times,1}\Q_p^{\times}$ by letting $\Q_p^{\times}$ act via the character $\chi|\cdot|^{b+2a}(\e/\eta^2)^i$ where $c-b=2a+i$ with $a\in\Z$, $i\in\{0,1\}$. It's easy to check that as a $D^{\times,1}\Q_p^{\times}$-representation, we have
\[
W\cong \Sym^b E^2\otimes\Nrd^a\otimes V\otimes\eta^i\circ\Nrd.
\]
Then we have an isomorphism
\[ \Ind^{D^{\times}}_{D^{\times,1}\Q_p^{\times}}(\Sym^b E^2\otimes\Nrd^a\otimes V\otimes\eta^i\circ\Nrd)
\cong
\Sym^b E^2\otimes\Nrd^a\otimes\eta^i\circ\Nrd\otimes\Ind^{D^{\times}}_{D^{\times,1}\Q_p^{\times}}V\] 
and a non-zero surjective morphism $\sum\Ind^{D^{\times}}_{D^{\times,1}\Q_p^{\times}}W\to\Pi^{\onealg}$. Since $D^{\times,1}\Q_p^{\times}$ has finite index in $D^{\times}$, the representation $\Ind^{D^{\times}}_{D^{\times,1}\Q_p^{\times}}W$ is semi-simple. So we have the first assertion.

By \cite[Proposition 6.13]{Paskunas-JL}, every finite dimensional representation of $D^{\times,1}$ is locally algebraic. So we have the last assertion.
\end{proof}
\begin{theorem}\label{algfin}The subspace of locally $D^{\times,1}$-algebraic vectors $\check{\mathcal{S}}^1(\Pi)^{\onealg}$ is finite dimensional in the following cases:
\hfill\begin{enumerate}
\item
$\Pi$ is a unitary admissible irreducible non-ordinary Banach space representation of $\GL_2(\mathbb{Q}_p)$;
\item $\Pi\cong(\Ind_B^G\d_2\e\otimes\d_1)_{\rm cont}$ with unitary characters $\d_1,\d_2\colon\Q^{\times}_p\to E^{\times}$ such that  $\d_1/\d_2\not\equiv1\pmod\varpi$. 
\end{enumerate}
\end{theorem}
\begin{proof}
Since twisting by a character does not change the $D^{\times,1}$-action, we can assume the central character of $\check{\mathcal{S}}^1(\Pi)$ is locally algebraic. By Lemma \ref{findim} and the admissibility of $\check{\mathcal{S}}^1(\Pi)$, it's enough to show that the number of isomorphism classes of finite-dimensional irreducible $D^{\times}$-subrepresentations of $\check{\mathcal{S}}^1(\Pi)$ is finite. Suppose $\Sym^b E^2\otimes \det^a\otimes V\otimes\eta^{i}\circ\Nrd$ is one of such subspaces. Then $\Sym^b E^2\otimes \Nrd^a\otimes V$ is a sub-representation of $\check{\mathcal{S}}^1(\Pi\otimes\eta^{-i}\circ\det)$.  Let $\sigma'_{\rm sm}$ be an irreducible $\cO_{D}^{\times}$-invariant subspace of $V$. We have 
\[\Hom_{\cO^{\times}_{D}}(\Sym^b E^2\otimes\Nrd^a\otimes \sigma'_{\rm sm},\check{\cS}^1(\Pi\otimes\eta^{-i}\circ\det))\neq0.\]
By Corollary \ref{O_Dalg}, $b$, $a$ and $\sigma'_{\rm sm}$ (up to conjugation by the uniformizer of $D$) are determined by $\Pi$.
Since $V$ is a quotient of $\Ind_{\Q_p^{\times}\cO_{D}^{\times}}^{D^{\times}}\sigma'_{\rm sm}$, there are at most two such $V$'s containing $\sigma'_{\rm sm}$. This completes the proof of the theorem. \end{proof}

\section{Local arguments}\label{local}
\subsection{Multiplicity and the category $\cC$}
Let $\varpi_D$ be a uniformizer of $D$ such that $\varpi^2_D=p$. Let $U^1_D:=1+\varpi_D\cO_D$ and let $Z_D$ be the center of $D^{\times}$. We write $Z^1_D$ for the subgroup $Z_D\cap U^1_D$. Let $\fm_D$ be the maximal ideal of the Iwasawa algebra $\L:=\F[\![U^1_D/Z^1_D]\!]$. The graded ring
\[
\gr_{\fm_D}(\L):=\bigoplus_{n\ge0}\fm^n_D/\fm^{n+1}_D
\]
of $\L$ is isomorphic to the universal enveloping algebra of the Lie algebra $\F y\oplus\F z\oplus\F h$ with the relations
\[
[y, z]=h,\:[y, h]=[z,h]=0
\] 
and $\deg y=\deg z=1$, where $y=\overline{Y}$, $z=\overline{Z}$ and $Y, Z\in \L$ are defined in \cite[Definition 2.16]{HW-JL2}. Let $J$ be the two-sided ideal generated by $yz$ and $zy$. Then we have $\gr_{\fm_D}(\L)/J\cong\F[y,z]/(yz)$. The ring $A=\F[y,z]/(yz)$ has two minimal ideals $\p_0=(y)$ and $\p_1=(z)$. Clearly $A_{\p_0}\cong\F(z)$ and $A_{\p_1}\cong\F(y)$. If $N$ is a finitely generated graded $\gr_{\fm_D}(\L)$-module annihilated by $J^n$ for some $n \ge 1$ and $\q$ is a minimal graded prime ideal of $\gr_{\fm_D}(\L)/J$, we define the multiplicity of $N$ at $\q$ to be
\[m_{\q}(N) =\sum_{i=0}^n l_{\q}(J^iN/J^{i+1}N),\]
where $l_{\q}(\cdot)$ is the length of $(\cdot)_{\q}$ over the ring $(\gr_{\fm_D}(\L)/J)_{\q}$.

Denote by $\cC$ the category of admissible smooth $\F$-representations $\pi$ of $D^{\times}$ with a central character, such that for some (equivalently any) good filtration (see \cite[\S I.5]{LiO} for the definition) $F$ on $\pi^{\vee}$ the graded module $\gr_F (\pi^{\vee})$ is annihilated by a finite power of $J$. Here we view $\pi^{\vee}$ as a (finitely generated) $\F[\![U^1_D/Z^1_D]\!]$-module and the filtration $F$ on $\pi^{\vee}$ is compatible with the $\fm_{D}$-adic filtration on $\F[\![U^1_D/Z^1_D]\!]$. It is clear that $\cC$ is an abelian category and is stable under subquotients and extensions. For $\pi\in\cC$, we define the multiplicity of $\pi$ as 
\[
\mu(\pi):=m_{\p_0}(\gr_{\fm_D}(\pi^{\vee}))+m_{\p_1}(\gr_{\fm_D}(\pi^{\vee})).
\]
By \cite[Lemma 3.1.4.3 and Lemma 3.3.4.4]{BHHMS2}, $\mu$ is additive on $\cC$. Clearly if $\pi$ is finite dimensional, then $\pi\in\cC$ and $\m(\pi)=0$.  Conversely if $\pi\in\cC$ and $\m(\pi)=0$, then each $J^i\gr_{\fm_D}(\pi^{\vee})/J^{i+1}\gr_{\fm_D}(\pi^{\vee})$ is finite dimensional. This implies that $\pi$ is also finite dimensional.

Let $\Pi$ be an admissible unitary Banach space representation of $D^{\times}$ over $E$ with a central character $\z$. We say $\Pi\in\widehat{\cC}$ if $\Theta/\varpi\Theta\in\cC$ for some open bounded $D^{\times}$-invariant lattice $\Theta$ in $\Pi$. We define $\mu(\Pi):=\mu(\Theta/\varpi\Theta)$. By \cite[Lemma 3.11]{HW-JL2}, these definitions do not depend on the choice of $\Theta$. Clearly $\widehat{\cC}$ is a subcategory of $\BanDadm$ stable under subquotients and extensions.
\begin{lemma}\label{SinC}
Let $\Pi$ be a unitary admissible Banach space representation of $G$. Let $\Theta$ be a $G$-invariant $\cO$-lattice in $\Pi$. Suppose that
\begin{equation}\label{dimless1}\dim_{\F}\cS^2((\Theta/\varpi)^{\rm ss})<\infty,\; \dim_{\F}\cS^0((\Theta/\varpi)^{\rm ss})<\infty.\end{equation} 
If $\cS^1(\Theta/\varpi\Theta)\in\cC$, then $\Shat^1(\Pi)\in\widehat{\cC}$ and $\mu(\Shat^1(\Pi))=\mu(\cS^1(\Theta/\varpi))$.
\end{lemma}
\begin{proof}
We write $\Shat^i(\Theta^d)_{\rm tor}$ for the $\cO$-torsion part of $\Shat^i(\Theta^d)$ and write $\Shat^i(\Theta^d)_{\rm tf}$ for the maximal $\cO$-torsion free quotient of $\Shat^i(\Theta^d)$. Since $\Shat^i(\Theta^d)$ is finitely generated over the noetherian profinite ring $\cO[\![\cO^{\times}_{D}]\!]$ by the proof of \cite[Lemma 3.4]{Paskunas-JL}, $\Shat^i(\Theta^d)_{\rm tor}$ is also finitely generated over $\cO[\![\cO^{\times}_{D}]\!]$. Therefore there exists some $m\ge1$ such that $\Shat^i(\Theta^d)_{\rm tor}=\Shat^i(\Theta^d)[\varpi^m]$ for $0\le i\le 2$. We have an exact sequence 
\[
0\to\Shat^i(\Theta^d)_{\rm tor}\to\Shat^i(\Theta^d)/\varpi^m\to\Shat^i(\Theta^d)_{\rm tf}/\varpi^m\to0.
\]
Since $\Shat^0$ is covariant right exact, $\Shat^0(\Theta^d)/\varpi^m\cong\Shat^0(\Theta^d/\varpi^m)$. An easy induction argument shows that $l_{\cO}(\Shat^0(\Theta^d/\varpi^m))<\infty$. Then it follows from $\Shat^0(\Theta^d)[\varpi]\subset\Shat^0(\Theta^d)_{\rm tor}\subset\Shat^0(\Theta^d)/\varpi^m$ that $\Shat^0(\Theta^d)[\varpi]$ is finite dimensional.

 By \cite[(14)]{Paskunas-JL}, there is an exact sequence
\[0\to\Shat^1(\Theta^d)/\varpi\to\Shat^1(\Theta^d/\varpi )\to\Shat^0(\Theta^d)[\varpi]\to0.\] 
Hence $(\Shat^1(\Theta^d)/\varpi)^{\vee}$ is a quotient of $\cS^1(\Theta/\varpi )$ with a finite dimensional kernel, which implies $(\Shat^1(\Theta^d)/\varpi)^{\vee}\in\cC$ and $\m((\Shat^1(\Theta^d)/\varpi)^{\vee})=\mu(\cS^1(\Theta/\varpi))$.

By \cite[(14)]{Paskunas-JL} with $\varpi$ replaced by $\varpi^m$, we have an exact sequence
\[0\to\Shat^2(\Theta^d)/\varpi^m\to\Shat^2(\Theta^d/\varpi^m )\to\Shat^1(\Theta^d)[\varpi^m]\to0.\] 
Then we have 
\[
l_{\cO}(\Shat^1(\Theta^d)_{\rm tor}/\varpi)\le l_{\cO}(\Shat^1(\Theta^d)_{\rm tor})=l_{\cO}(\Shat^1(\Theta^d)[\varpi^m])\le l_{\cO}(\Shat^2(\Theta^d/\varpi^m ))< \infty.
\]
It then follows from the exact sequence
\[
0\to\Shat^1(\Theta^d)_{\rm tor}/\varpi\to\Shat^1(\Theta^d)/\varpi\to\Shat^1(\Theta^d)_{\rm tf}/\varpi\to0
\]
that $(\Shat^1(\Theta^d)_{\rm tf}/\varpi)^{\vee}$ is a subrepresentation of $(\Shat^1(\Theta^d)/\varpi)^{\vee}$ with finite codimension. Hence $(\Shat^1(\Theta^d)_{\rm tf}/\varpi)^{\vee}\in\cC$ and $\m((\Shat^1(\Theta^d)_{\rm tf}/\varpi)^{\vee})=\mu(\cS^1(\Theta/\varpi))$. Since $\Shat^1(\Theta^d)^d/\varpi \cong(\Shat^1(\Theta^d)_{\rm tf}/\varpi )^{\vee}$, we have $\Shat^1(\Pi)\in\widehat{\cC}$ and $\mu(\Shat^1(\Pi))=\mu(\cS^1(\Theta/\varpi))$. \end{proof}
\subsection{A finiteness criterion}
One of the main results of \cite{HW-JL2} is the nonexistence of $\Pi\in\widehat{\cC}$ with multiplicity $2$. 
\begin{theorem}\label{mu>2}
If $\pi\in\cC$, then $\mu(\pi)\in2\Z_{\ge0}$. If moreover $\pi$ is infinite dimensional, then $\mu(\pi)\ge4$.
\end{theorem}
\begin{proof}
This is \cite[Lemma 3.6 and Theorem 3.10]{HW-JL2}.
\end{proof}
We follow the argument of the proof of \cite[Theorem 3.13]{HW-JL2} to deduce the following theorem. However, in {\em loc. cit.} the space $\Pi^{\rm lalg}$ should be replaced by $\Pi^{\onealg}$, since finite dimensional $D^{\times}$-representations are not necessarily locally algebraic as $D^{\times}$-representations.
\begin{theorem}\label{fincri}
Suppose $\Pi\in\widehat{\cC}$ and that $\mu(\Pi)\le4$. If $\Pi^{\onealg}$ is finite dimensional, then $\Pi$ is topologically of finite length.
\end{theorem}
\begin{proof}
It's enough to assume that $\Pi$ is infinite dimensional. Then $\Pi/\Pi^{\onealg}$ is infinite dimensional. Since $\Pi/\Pi^{\onealg}$ is admissible, there is a nonzero irreducible subrepresentation $\Pi'$ of $\Pi/\Pi^{\onealg}$. Since $\widehat{\cC}$ is stable under subquotient and $\mu$ is additive on $\widehat{\cC}$, we have $\Pi'\in\widehat{\cC}$ and $\mu(\Pi')\le\mu(\Pi)\le4$. By Lemma \ref{findim}, $\Pi'$ is  infinite dimensional. It follows from Theorem \ref{mu>2} that $\mu(\Pi')\ge4$. Hence $\mu(\Pi')=\mu(\Pi)=4$ and $\mu((\Pi/\Pi^{\onealg})/\Pi')=0$. So $\Pi'$ has finite codimension in $\Pi/\Pi^{\onealg}$.
\end{proof}

\section{Main result}\label{main}
\begin{lemma}\label{modpresult}
Let $\overline{\rho}\colon G_{\Q_p}\to\GL_2(\F)$ be a continuous representation. 
\hfill\begin{enumerate}
\item If $\overline{\rho}$ is absolutely irreducible, we assume $\overline{\rho}$ is generic in the sense of Definition \ref{generic}; 
\item If $\overline{\rho}\cong\big(\begin{smallmatrix}\chi_1& *\\0&\chi_2\end{smallmatrix}\big)$ is a nonsplit extension of $\chi_2$ by $\chi_1$, we assume $\chi_1\chi_2^{-1}|_{I_{\Q_p}}\neq1,\o$.
\end{enumerate}
Then $\cS^1(\pi(\overline{\rho}))\in\cC$ and $\mu(\cS^1(\pi(\overline{\rho})))=8$.
\end{lemma}
\begin{proof}
We use the notation of  \cite{HW-JL1}. It follows from \cite[Theorem 7.7]{HW-JL1} that 
\begin{equation}\label{multd}
\pi^{B'}(\overline{\rho})\cong\pi(\overline{\rho})^d
\end{equation} for some $d\ge1$. And we have a $G_{\Q_p}\times D^{\times}$-equivariant inclusion
\begin{equation}\label{Ssubsetpi}
\cS^1(\pi^{B'}(\overline{\rho}))\subset \overline{\rho}\otimes\pi^{B}(\overline{\rho})
\end{equation}
by \cite[Proposition 7.6]{HW-JL1}. The cokernel of this inclusion is finite dimensional and $\cO^{\times}_{D}\cap D^{\times,1}$ acts trivially on the cokernel. Let $W_D(\overline{\rho}\otimes\o)$ be the quaternionic Serre weights for $\overline{\rho}\otimes\o$ which is denoted by $W^?(\overline{\rho}\otimes\o)$ in \cite[Definition 3.4]{MR2822861}. 
According to \cite[Theorem 6.14]{HW-JL1}, we have $\pi^B(\overline{\rho})\in\cC$ and 
\begin{equation}\label{mu4m}\mu(\pi^B(\overline{\rho}))\le4m\end{equation}
where $m=\dim_{\F}\Hom_{\cO^{\times}_D}(\chi,\pi^{B}(\overline{\rho}))$ for each $\chi\in W_D(\overline{\rho}\otimes\o)$. By our assumption on $\overline{\rho}$ and description of $W_D(\overline{\rho}\otimes\o)$ in  \cite[Proposition 6.1]{HW-JL1}), for each $\chi\in W_D(\overline{\rho}\otimes\o)$, the group $\chi(\cO^{\times}_{D}\cap D^{\times,1})$ is non-trivial. So we have
\[
\begin{aligned}
2m&=\dim_{\F}\Hom_{\cO^{\times}_D}(\chi,\overline{\rho}\otimes\pi^{B}(\overline{\rho}))\\
&=\dim_{\F}\Hom_{\cO^{\times}_D}(\chi,\cS^1(\pi^{B'}(\overline{\rho})))\\
&=\dim_{\F}\Hom_{\cO^{\times}_D}(\chi,\cS^1(\pi(\overline{\rho}))^{\oplus d})\\
\end{aligned}
\] 

When $\overline{\rho}$ is absolutely irreducible and generic or $\overline{\rho}\cong\big(\begin{smallmatrix}\chi_1& *\\0&\chi_2\end{smallmatrix}\big)$ with $\chi_1\chi_2^{-1}\neq \o^{-1}$, we have $\dim_{\F}\Hom_{\cO^{\times}_D}(\chi,\cS^1(\pi(\overline{\rho})))=2$ by \cite[Theorem 4.18]{HW-JL2}. When $\overline{\rho}\cong\big(\begin{smallmatrix}\chi_1& *\\0&\chi_2\end{smallmatrix}\big)$ with $\chi_1\chi_2^{-1}= \o^{-1}$, we still have $\dim_{\F}\Hom_{\cO^{\times}_D}(\chi,\cS^1(\pi(\overline{\rho})))=2$  by \cite[Proposition 8.20 and Corollary 8.30]{HW-JL1}. Therefore we have $m=d$, which implies $\mu(\pi^B(\overline{\rho}))\le4$ by (\ref{multd}), (\ref{Ssubsetpi})and (\ref{mu4m}). Since $\pi^B(\overline{\rho})$ is infinite dimensional (\cite[Theorem 7.8]{Scholze}), we conclude that $\mu(\pi^B(\overline{\rho}))=4$ and $\mu(\cS^1(\pi(\overline{\rho})))=8$.
\end{proof}
\begin{corollary}\label{mu4}
For all continuous characters $\chi_1,\chi_2\colon G_{\Q_p}\to \F^{\times}$ with $\chi_1\chi_2^{-1}|_{I_{\Q_p}}$ non-trivial, we have $\cS^1(\Ind_B^G\chi_1\o\otimes\chi_2)\in\cC$ and $\mu(\cS^1(\Ind_B^G\chi_1\o\otimes\chi_2))=4$.
\end{corollary}
\begin{proof}
Let  $\overline{\rho}\cong\big(\begin{smallmatrix}\chi_1& *\\0&\chi_2\end{smallmatrix}\big)$ be a nonsplit extension of $\chi_2$ by $\chi_1$. If $\chi_1\chi_2^{-1}|_{I_{\Q_p}}\neq1,\o$, it follows from Lemma \ref{modpresult} that $\cS^1(\pi(\overline{\rho}))\in\cC$ and $\mu(\cS^1(\pi(\overline{\rho})))=8$. If $\chi_1\chi_2^{-1}|_{I_{\Q_p}}=\o$, let  $\overline{\rho}_1\cong\big(\begin{smallmatrix}\chi_2& *\\0&\chi_1\end{smallmatrix}\big)$ be a nonsplit extension of $\chi_1$ by $\chi_2$. Then $\cS^1(\pi(\overline{\rho}_1))\in\cC$ and $\mu(\cS^1(\pi(\overline{\rho}_1)))=8$. Since $\pi(\overline{\rho})$ and $\pi(\overline{\rho}_1)$ have the same Jordan–H\"older factors (up to finite dimensional representations), we still have $\cS^1(\pi(\overline{\rho}))\in\cC$ and $\mu(\cS^1(\pi(\overline{\rho})))=8$. For any character $\chi_1,\chi_2\colon G_{\Q_p}\to \F^{\times}$, 
\begin{equation}\label{twoterm}
\cS^1(\Ind_B^G\chi_1\o\otimes\chi_2) ,\;\cS^1(\Ind_B^G\chi_2\o\otimes\chi_1)
\end{equation} 
are infinite dimensional by the proof of \cite[Theorem 8.12]{HW-JL1} and \cite[Theorem 8.33]{HW-JL1}. Hence both of the two terms of (\ref{twoterm}) are objects in $\cC$ with multiplicity $4$ when $\chi_1\chi_2^{-1}|_{I_{\Q_p}}$ is non-trivial.
\end{proof}
\begin{theorem}
Let $\Pi\cong(\Ind_B^G\d_2\e\otimes\d_1)_{\rm cont}$ with unitary characters $\d_1,\d_2\colon\Q^{\times}_p\to E^{\times}$ such that  $\d_1\d_2^{-1}|_{\Z_p^{\times}}\not\equiv1\pmod\varpi$. Then $\Shat^1(\Pi)$ is  infinite dimensional and is topologically of finite length.
\end{theorem}
\begin{proof}
This follows from Theorem \ref{fincri} combined with Theorem \ref{algfin}, Lemma \ref{SinC} and Corollary \ref{mu4}.
\end{proof}

\begin{theorem}\label{mult}
Let $\rho\colon G_{\Q_p}\to\GL_2(E)$ be a continuous absolutely irreducible representation. Suppose
\hfill\begin{enumerate}
\item $\overline{\rho}$ is absolutely irreducible and is generic in the sense of Definition \ref{generic} or
\item $\overline{\rho}^{\rm ss}\cong\chi_1\oplus\chi_2$ with $\chi_1\chi_2^{-1}|_{I_{\Q_p}}$ non-trivial.  
\end{enumerate}
Then $\check{\mathcal{S}}^1(\Pi(\rho))$ is infinite dimensional and is topologically of finite length.
\end{theorem}
\begin{proof}
It follows from Corollary \ref{S(Pi)typic} that 
\[\Shat^1(\Pi(\rho))\cong\rho\boxtimes {\rm JL}(\rho)
\]
for a unitary Banach representation ${\rm JL}(\rho)$ of $D^{\times}$. And we have $\mu(\JL(\rho))=4$ by Lemma \ref{modpresult} (for $\overline{\rho}$  absolutely irreducible) and Corollary \ref{mu4} (for $\overline{\rho}$ reducible). Since $\JL(\rho)^{\onealg}$ is finite dimensional, the theorem follows from Theorem \ref{fincri}.
\end{proof}

\bibliography{references}

\bibliographystyle{plain}

\bigskip

\noindent  Department of Mathematical Sciences, Tsinghua University, Beijing, 100084\\
{\it E-mail:} {\ttfamily liuhao21@mails.tsinghua.edu.cn}\\

\noindent  Academy for Multidisciplinary Studies, Capital Normal University, Beijing, 100048\\
{\it E-mail:} {\ttfamily haoran@cnu.edu.cn}\\

 \end{document}